\documentclass[a4paper,11pt,final]{article}

\usepackage{graphicx}
\usepackage{amssymb,amsmath,amsthm,mathrsfs,xfrac}
\usepackage{enumerate}
\usepackage{a4wide}
\usepackage{hyperref}
\usepackage{array}
\usepackage{color}

\numberwithin{equation}{section}
\allowdisplaybreaks[4]

\newtheorem{proposition}{Proposition}[section]
\newtheorem{theorem}[proposition]{Theorem}
\newtheorem{lemma}[proposition]{Lemma}
\newtheorem{definition}[proposition]{Definition}

\newtheorem{remark}[proposition]{Remark}

\renewenvironment{proof}{\smallskip\noindent{\textbf{Proof.}}%
  \hspace{1pt}}{\hspace{-5pt}{\nobreak\quad\nobreak\hfill\nobreak%
    $\square$\vspace{2pt}\par}\smallskip\goodbreak}

\newenvironment{proofof}[1]{\smallskip\noindent{\textbf{Proof~of~#1.}}%
  \hspace{1pt}}{\hspace{-5pt}{\nobreak\quad\nobreak\hfill\nobreak%
    $\square$\vspace{2pt}\par}\smallskip\goodbreak}

\newcommand{\C}[1]{\mathbf{C}^{#1}}
\newcommand{\Cc}[1]{\mathbf{C}_c^{#1}}

\newcommand{\BV}{\mathbf{BV}}

\renewcommand{\L}[1]{{\mathbf{L}^#1}}

\newcommand{\Wloc}[2]{{\mathbf{W}_{\mathbf{loc}}^{#1,#2}}}
\newcommand{\modulo}[1]{{\left|#1\right|}}
\newcommand{\norma}[1]{{\left\|#1\right\|}}
\newcommand{\caratt}[1]{{\chi_{\strut#1}}}
\newcommand{\reali}{{\mathbb{R}}}
\newcommand{\naturali}{{\mathbb{N}}}

\renewcommand{\epsilon}{\varepsilon}
\renewcommand{\phi}{\varphi}
\renewcommand{\theta}{\vartheta}

\newcommand{\tv}{\mathinner{\rm TV}}
\newcommand{\spt}{\mathop{\rm spt}}
\newcommand{\sgn}{\mathop{\rm sgn}}

\renewcommand{\d}[1]{\mathinner{\mathrm{d}{#1}}}

\newcommand{\ub}{\boldsymbol{u_b}}

\DeclareMathOperator*{\esssup}{ess\,sup}
\DeclareMathOperator*{\essinf}{ess\,inf}

\makeatletter
\let\@fnsymbol\@arabic
\makeatother

\title{IBVPs for Scalar Conservation Laws \\ with Time Discontinuous
  Fluxes}

\author{Rinaldo M.~Colombo\footnote{\texttt{rinaldo.colombo@unibs.it}}
  \qquad Elena Rossi\footnote{\texttt{elena.rossi@unibs.it}} \\ INDAM
  Unit, University of Brescia, Italy}
\date{}

\begin{document}

\maketitle

\begin{abstract}

  \noindent The initial boundary value problem for a class of scalar
  non autonomous conservation laws in one space dimension is proved to
  be well posed and stable with respect to variations in the
  flux. Targeting applications to traffic, the regularity assumptions
  on the flow are extended to a merely $\L\infty$ dependence on
  time. These results ensure, for instance, the well posedness of a
  class of vehicular traffic models with time dependent speed
  limits. A traffic management problem is then shown to admit an
  optimal solution.

  \medskip

  \noindent\textit{2010~Mathematics Subject Classification:} 35L65,
  35L04

  \medskip

\end{abstract}

\section{Introduction}
\label{sec:Intro}

In this paper we deal with a non linear Initial Boundary Value Problem
(IBVP) for a non autonomous scalar conservation law in one space
dimension. Our main result is its well posedness and the stability of
solutions with respect to variations in the flux, relaxing the
regularity assumptions found in the literature, see for
instance~\cite{BardosLerouxNedelec, IBVP1D, Martin, Vovelle}.

The theory of Conservation Laws traditionally splits in that of scalar
multi--dimensional equations and that of one dimensional systems.  In
the former case, the key reference related to IBVPs
is~\cite{BardosLerouxNedelec}, see also~\cite{bordo, IBVP1D,
  MalekEtAlBook, Martin, OttoCR, Vovelle}. In the latter case, we
refer to~\cite{AmadoriColombo97, DonadelloMarson,
  DuboisLeFloch, Goodman}.

Below, we consider the following IBVP both on the (unbounded) half
line $\reali_+$
\begin{equation}
  \label{eq:1part}
  \left\{
    \begin{array}{l@{\,}c@{\,}lr@{\,}c@{\,}l}
      \multicolumn{3}{l}{%
      \partial_t u + \partial_x \left(v(t) \, g(u)\right) = 0\qquad\ }
      & (t,x)
      & \in
      & [0,T] \times \reali_+
      \\
      u (0,x)
      & =
      & u_o (x)
      & x
      & \in
      & \reali_+
      \\
      u (t, 0)
      & =
      & u_b (t)
      & t
      & \in
      & [0,T],
    \end{array}
  \right.
\end{equation}
and on the bounded segment $[0, L]$
\begin{equation}
  \label{eq:1sgPart}
  \left\{
    \begin{array}{l@{\,}c@{\,}lr@{\,}c@{\,}l}
      \multicolumn{3}{l}{%
      \partial_t u + \partial_x \left(v(t) \, g(u)\right) = 0\qquad\ }
      & (t,x)
      & \in
      & [0,T] \times [0,L]
      \\
      u (0,x)
      & =
      & u_o (x)
      & x
      & \in
      & [0,L]
      \\
      u (t, 0)
      & =
      & u_{b,1} (t)
      &  t
      & \in
      & [0,T]
      \\
      u (t,L)
      & =
      & u_{b,2} (t)
      & t
      & \in
      & [0,T].
    \end{array}
  \right.
\end{equation}
Remark that here the flux is the product between a time dependent
function $v (t)$ and a function of the unknown $g (u)$. The time
dependent part of the flux is here assumed to be merely $\L\infty$ and
the domain needs not be bounded, significantly extending what required
in~\cite{BardosLerouxNedelec, Martin, Vovelle}, where the flow is
smooth and the domain is bounded. In~\cite{IBVP1D}, the domain can be
unbounded but the flow needs be smooth too, although of the more
general form $f (t,u)$.  Nevertheless, the results in~\cite[\S~2.1 and
\S~3.1]{IBVP1D} concerning the autonomous (time independent) problems,
that is with flux of the form $f (u)$, constitute the starting point
for the present work. Aiming to deal with a merely $\L\infty$ function
$v$, as an intermediate step we first focus on the case of Lipschitz
continuous $v$. Under this assumption, an \emph{ad hoc}
diffeomorphism, transforming~\eqref{eq:1part},
respectively~\eqref{eq:1sgPart}, into an autonomous problem of the
type considered in~\cite{IBVP1D}, allows to extend the well posedness
and stability results to the present setting. The stability with
respect to the flux, and in particular to $v$, allows then to further
relax the regularity assumptions on $v$, up to $\L\infty$, without any
requirement on the boundedness of the total variation in time.

We adopt the same definition of solution to~\eqref{eq:1part},
respectively~\eqref{eq:1sgPart}, as in~\cite{IBVP1D}.  As far as
bounded domains are concerned, this definition was introduced
in~\cite{Martin, OttoCR, Vovelle}, see also~\cite{MalekEtAlBook}. As
in~\cite{IBVP1D}, we also consider the unbounded case of the half
line. This definition of solution has the remarkable feature of being
stable under convergence in $\L1$, see~\cite[Chapter~2,
Remark~7.33]{MalekEtAlBook} and also~\cite[Remark~3.6]{Elena1}. This
constitutes a crucial element in the proof of the well posedness of
problem~\eqref{eq:1part} when $v$ is a function in $\L\infty$, see
Theorem~\ref{thm:main}. We refer to~\cite{Elena1} for a deeper
discussion of the various definitions of solutions to IBVPs for
balance laws.

The well posedness and stability results obtained below suits various
issues arising in the macroscopic modelling of vehicular traffic
flows. Indeed, we consider in detail the problem of choosing a
variable speed limit that minimizes the time spent in queues at a
traffic light. We refer to~\cite{Strnad2016841} and to the references
therein for a more engineering oriented approach, to~\cite{simai} for
further related control and inverse problems.

\smallskip

The paper is organised as follows. Section~\ref{sec:Main} deals with
problem~\eqref{eq:1part} on the half line: after providing the
definition of solution, \S~\ref{sec:lipschitz} presents the results
for Lipschitz continuous $v$, while \S~\ref{sec:discont} is devoted to
the case $v \in \L\infty$. Section~\ref{sec:seg} is structured in a
similar way for the IBVP~\eqref{eq:1sgPart} on a bounded
segment. These results allow to prove the existence of an optimal
control for a traffic management problem, see
Section~\ref{sec:example}. All proofs related to the case of the half
line are collected in Section~\ref{sec:TP}, where the necessary
preliminary results are also recalled.

\section{The Case of a Half Line}
\label{sec:Main}

All statements and proofs are referred to the time interval $[0,T]$,
for a fixed positive $T$. Throughout, $v_{\min}$ is a strictly
positive constant.

Below, if $u_\ell \in \L\infty (I_\ell;\reali)$ for real intervals
$I_\ell$, $\ell=1, \ldots, m$, we denote
\begin{equation}
  \label{eq:19}
  \mathcal{U} (u_1, \ldots, u_m)
  =
  \bigl[
  \min_{\ell=1, \ldots, m} \essinf_{I_\ell} u_\ell \,, \;
  \max_{\ell=1, \ldots, m} \esssup_{I_\ell} u_\ell
  \bigr] \,.
\end{equation}
In other words, $\mathcal{U} (u_1, \ldots, u_m)$ is the closed convex
hull of $\bigcup_{\ell=1}^m u_\ell (I_\ell)$. Whenever $I_u$ is a real
interval, for $u \in \BV (I_u; \reali)$, $\tv (u)$ stands for the
total variation of $u$ on $I_u$, see~\cite[\S~5.10.1]{EvansGariepy},
and, for any interval $I \subseteq I_u$, we also set
$\tv (u;I) = \tv (u_{|I})$. For the vector
$\boldsymbol{u}= \left(u_1, \ldots, u_m\right)$, we define
$ \tv(\boldsymbol{u}) = \sum_{\ell=1}^m \tv (u_\ell)$.

The following notation is of use
\begin{align*}
  \sgn{}^+ (s) = \
  & \begin{cases}
    1 &\mbox{ if } s>0,\\
    0 &\mbox{ if } s \leq 0,
  \end{cases}
      &
        \sgn{}^- (s) = \
      & \begin{cases}
        0 &\mbox{ if } s\geq0,\\
        -1 & \mbox{ if } s < 0,
      \end{cases}
      &
        \begin{aligned}
          s^+= \ & \max\{s,0\},\\
          s^-= \ & \max\{-s,0\}.
        \end{aligned}
\end{align*}
Occasionally, we also denote $t_1 \vee t_2 = \max\{t_1, t_2\}$.

\smallskip

\begin{definition}
  \label{def:solhl}
  A \emph{solution} to the IBVP
  \begin{equation}
    \label{eq:1}
    \left\{
      \begin{array}{l@{\,}c@{\,}lr@{\,}c@{\,}l}
        \multicolumn{3}{l}{%
        \partial_t u + \partial_x f (t,u) = 0\qquad\ }
        & (t,x)
        & \in
        & [0,T]  \times  \reali_+
        \\
        u (0,x)
        & =
        & u_o (x)
        & x
        & \in
        & \reali_+
        \\
        u (t, 0)
        & =
        & u_b (t)
        & t
        & \in
        & [0,T],
      \end{array}
    \right.
  \end{equation}
  is a map $u \in \L\infty ([0,T] \times \reali_+; \reali)$ such that
  for any $k \in \reali$ and for any test function
  $\phi \in \Cc1 (\reali \times \reali; \reali_+)$
  \begin{align*}
    & \int_0^T \!\!\! \int_{\reali_+}
      \left\{
      \left(u (t,x) - k\right)^\pm  \partial_t\phi (t,x)
      +
      \sgn{}^\pm \!\left(u (t,x) - k \right)\!
      \left(f \left(t, u (t,x)\right) - f (t,k)\right)  \partial_x \phi (t,x)
      \right\}
      \d{x} \d{t}
    \\
    & + \int_{\reali_+} \left(u_o (x) - k\right)^\pm  \phi (0, x) \,
      \d{x}
      - \int_{\reali_+} \left(u (T,x) - k\right)^\pm \phi (T, x) \,
      \d{x}
    \\
    & + \norma{\partial_u f}_{\L\infty ([0,T] \times \mathcal{U}; \reali)}
      \int_0^T \left(u_b (t) - k\right)^\pm \phi (t,0) \, \d{t}
      \geq  0,
  \end{align*}
  where $\mathcal{U} = \mathcal{U} (u_o, {u_b}_{|[0,T]})$ as
  in~\eqref{eq:19}.
\end{definition}

The above choice is inspired by~\cite[Definition~2.1]{IBVP1D}, see
also~\cite{Martin, Vovelle} in a slightly different setting. Refer in
particular to~\cite{Elena1} for a comparison among various definitions
of solutions to IBVPs for general scalar balance laws, also in several
space dimensions.

\subsection{Lipschitz Continuous $\boldsymbol{v}$}
\label{sec:lipschitz}

This paragraph is devoted to the well posedness of
problem~\eqref{eq:1part}, under the assumptions that the time
dependent part of the flux $v$ is in
$\C{0,1} ([0,T]; [v_{\min}, +\infty[)$. All proofs of the results
presented below are deferred to \S~\ref{sec:vLip}.

First, we extend~\cite[Proposition~2.2]{IBVP1D}, see also
Proposition~\ref{prop:Elena}, to the present more general case,
i.e.~under less regularity assumptions on the flux, obtaining the
$\L1$--Lipschitz continuous dependence of the solutions on initial and
boundary data.

\begin{proposition}
  \label{prop:nonso}
  Let $v \in \C{0,1} ([0,T]; [v_{\min}, +\infty[)$. Fix
  $g \in \C1 (\reali; \reali)$. Let
  $u_o, \tilde u_o \in (\L1 \cap \BV) (\reali_+; \reali)$,
  $u_b, \tilde u_b \in (\L1 \cap \BV) ([0,T]; \reali)$. Assume that
  the IBVPs
  \begin{displaymath}
    \left\{
      \begin{array}{r@{\;}c@{\;}l}
        \multicolumn{3}{l}{\partial_t u + \partial_x \left(v(t) \, g(u)\right) = 0}
        \\
        u (0,x)
     & =
     & u_o (x)
        \\
        u (t, 0)
     & =
     & u_b (t)
      \end{array}
    \right.
    \quad \mbox{ and } \quad
    \left\{
      \begin{array}{r@{\;}c@{\;}l}
        \multicolumn{3}{l}{\partial_t \tilde u + \partial_x \left(v(t) \, g(\tilde u)\right) = 0}
        \\
        \tilde u (0,x)
     & =
     & \tilde u_o (x)
        \\
        \tilde u (t, 0)
     & =
     & \tilde u_b (t)
      \end{array}
    \right.
  \end{displaymath}
  admit solutions
  $u,\tilde u \in \L\infty ([0,T] \times \reali_+; \reali)$ in the
  sense of Definition~\ref{def:solhl}, such that $u$ and $\tilde u$
  both admit a trace for $x \to 0+$ for a.e.~$t \in [0,T]$. Then, for
  all $t \in [0,T]$,
  \begin{displaymath}
    \norma{u (t) - \tilde u (t)}_{\L1 (\reali_+; \reali)}
    \leq
    \norma{u_o - \tilde u_o}_{\L1 (\reali_+; \reali)}
    +
    \norma{v}_{\L\infty ([0,t];\reali)} \,
    \norma{g'}_{\L\infty (\mathcal{U};\reali)} \,
    \norma{u_b - \tilde u_b}_{\L1 ([0,t]; \reali)},
  \end{displaymath}
  where
  $\mathcal{U} = \mathcal{U}( {u_b}_{|[0,t]}, {\tilde u}_{b|[0,t]})$
  as in~\eqref{eq:19}.
\end{proposition}

\noindent Now, we deal with existence and \emph{a priori} estimates on
solutions to~\eqref{eq:1part} in the case of a merely Lipschitz
continuous $v$.

\begin{proposition}
  \label{prop:2part}
  Let $v \in \C{0,1} ([0,T]; [v_{\min}, +\infty[)$,
  $g \in \Wloc1\infty (\reali; \reali)$,
  $u_o \in (\L1 \cap \BV) (\reali_+; \reali)$ and
  $u_b \in (\L1 \cap \BV) ([0,T]; \reali)$. Then
  problem~\eqref{eq:1part} admits a solution $u$ in the sense of
  Definition~\ref{def:solhl} with the properties:
  \begin{enumerate}
  \item Range of $u$: with the notation in~\eqref{eq:19},
    $u (t,x) \in \mathcal{U} (u_o, {u_b}_{|[0,t]})$ for
    a.e.~$(t,x) \in [0,T] \times \reali_+$. Hence, for all
    $t\in [0,T]$,
    $ \norma{u (t)}_{\L\infty (\reali_+; \reali)} \leq \max \left\{
      \norma{u_o}_{\L\infty (\reali_+; \reali)}, \,
      \norma{u_b}_{\L\infty ([0,t];\reali)} \right\}$.
  \item $u$ is $\L1$--Lipschitz continuous in time: for all
    $t_1,t_2 \in [0,T]$,
    \begin{displaymath}
      \norma{u (t_1) - u (t_2)}_{\L1 (\reali_+; \reali)}
      \leq \tv (t_1 \vee t_2, u_o, u_b) \;
      \norma{v}_{\L\infty ([0,t_1 \vee t_2];\reali)} \,
      \norma{g'}_{\L\infty (\mathcal{U};\reali)} \,
      \modulo{t_2 - t_1},
    \end{displaymath}
    where
    \begin{equation}
      \label{eq:4}
      \tv (t, u_o, u_b)
      =
      \tv (u_o) + \tv (u_b; [0,t]) +\modulo{u_b (0+) - u_o (0+)},
    \end{equation}
    and $\mathcal{U} = \mathcal{U} (u_o, {u_b}_{|[0,t_1 \vee t_2]})$,
    with the notation~\eqref{eq:19}.
  \item Total variation estimate: for all $t \in [0,T]$,
    $\tv\left(u (t) \right) \leq \tv (t, u_o, u_b)$, with the
    notation~\eqref{eq:4}.
  \end{enumerate}
\end{proposition}

\noindent To conclude this paragraph, we ensure the stability with
respect to the flux of the solution to~\eqref{eq:1part}, under the
hypothesis $v \in \C{0,1} ([0,T]; [v_{\min}, +\infty[)$.

\begin{theorem}
  \label{thm:3part}
  Let $v, \tilde v \in \C{0,1} ([0,T]; [v_{\min}, +\infty[)$. Fix
  $g , \tilde g \in \C1 (\reali; \reali)$.  Let
  $u_o \in (\L1 \cap \BV) (\reali_+; \reali)$,
  $u_b \in (\L1 \cap \BV) ([0,T]; \reali)$. Call $u$ and $\tilde u$
  the solutions to the IBVPs
  \begin{equation}
    \label{eq:quella}
    \left\{
      \begin{array}{r@{\;}c@{\;}l}
        \multicolumn{3}{l}{\partial_t u + \partial_x \left(v(t) \, g(u)\right) = 0}
        \\
        u (0,x)
     & =
     & u_o (x)
        \\
        u (t, 0)
     & =
     & u_b (t)
      \end{array}
    \right. \quad \mbox{ and } \qquad
    \left\{
      \begin{array}{r@{\;}c@{\;}l}
        \multicolumn{3}{l}{\partial_t \tilde u + \partial_x \left(\tilde v(t) \, \tilde g(\tilde u)\right) = 0}
        \\
        \tilde u (0,x)
     & =
     & u_o (x)
        \\
        \tilde u (t, 0)
     & =
     & u_b (t).
      \end{array}
    \right.
  \end{equation}
  where $(t,x) \in [0,T] \times \reali_+$. Then the following
  estimate holds: for all $t \in [0,T]$
  \begin{equation}
    \label{eq:6}
    \norma{u (t)-\tilde u (t)}_{\L1 (\reali_+;\reali)}
    \leq
    \tv (t, u_o, u_b)
    \left(
      A \, t \, \norma{g' - \tilde g'}_{\L\infty (\mathcal{U};\reali)}
      +
      B  \, \norma{v-\tilde v}_{\L1 ([0,t];\reali)}
    \right),
  \end{equation}
  where $\tv (t, u_o, u_b)$ is defined in~\eqref{eq:4},
  $\mathcal{U} = \mathcal{U} (u_o, {u_b}_{\vert [0,t]})$ is as
  in~\eqref{eq:19} and
  \begin{equation}
    \label{eq:15}
    \begin{array}{r@{\;}l@{\qquad\qquad}r@{\;}l}
      A =
      & \max \left\{1,G\right\} \, V,
      & G =
      & \min
        \left\{
        \norma{g'}_{\L\infty (\mathcal{U};\reali)},
        \norma{\tilde g'}_{\L\infty (\mathcal{U};\reali)}
        \right\},
      \\
      B =
      & \left(1 + \dfrac{V}{v_{\min}}\right) \, G,
      & V =
      & \min \left\{
        \norma{v}_{\L\infty ([0,t];\reali)},
        \norma{\tilde v}_{\L\infty ([0,t];\reali)}
        \right\} .
    \end{array}
  \end{equation}
\end{theorem}

\begin{remark}
  {\rm In the case of a more regular flux function, say
    $v \in \C1 ([0,T]; [v_{\min}, +\infty[)$, also the stability
    estimate presented in~\cite[Theorem~2.6]{IBVP1D} can be applied,
    obtaining
    \begin{displaymath}
      \norma{u (t)-\tilde u (t)}_{\L1 (\reali_+;\reali)}
      \leq
      \tv (t, u_o, u_b)
      \left(
        A  \, \norma{g' - \tilde g'}_{\L\infty (\mathcal{U};\reali)}
        +
        B  \, \norma{v-\tilde v}_{\L\infty ([0,t];\reali)}
      \right) t,
    \end{displaymath}
    with the same notation introduced in~\eqref{eq:15}.  Clearly, the
    bound~\eqref{eq:6} is more precise. This improvement plays a key
    role in the relaxation of the hypothesis on $v$ achieved in
    Theorem~\ref{thm:main}.}
\end{remark}

\subsection{Discontinuous $\boldsymbol{v}$}
\label{sec:discont}

Our aim is now to further relax the regularity hypothesis on $v$,
allowing also for discontinuous functions. In particular, we set
$v \in \L\infty ([0,T] ; [v_{\min}, + \infty[)$.

\begin{theorem}\label{thm:main}
  Let $v \in \L\infty ([0,T] ; [v_{\min}, + \infty[)$,
  $g \in \C1 (\reali;\reali)$. Fix
  $u_o \in (\L1 \cap \BV) (\reali_+; \reali)$,
  $u_b \in (\L1 \cap \BV) ([0,T]; \reali)$. Then
  problem~\eqref{eq:1part} admits a solution $u$ in the sense of
  Definition~\ref{def:solhl} with the properties:
  \begin{enumerate}
  \item Range of $u$: with the notation in~\eqref{eq:19},
    $u (t,x) \in \mathcal{U} (u_o, {u_b}_{|[0,t]})$ for
    a.e.~$(t,x) \in [0,T] \times \reali_+$. Hence, for all
    $t\in [0,T]$,
    $ \norma{u (t)}_{\L\infty (\reali_+; \reali)} \leq \max \left\{
      \norma{u_o}_{\L\infty (\reali_+; \reali)}, \,
      \norma{u_b}_{\L\infty ([0,t];\reali)} \right\}$.
  \item $u$ is $\L1$--Lipschitz continuous in time: for all
    $t_1,t_2 \in [0,T]$,
    \begin{displaymath}
      \norma{u (t_1) - u (t_2)}_{\L1 (\reali_+; \reali)}
      \leq \tv (t_1 \vee t_2, u_o, u_b) \;
      \norma{v}_{\L\infty ([0,t_1 \vee t_2];\reali)} \,
      \norma{g'}_{\L\infty (\mathcal{U};\reali)} \,
      \modulo{t_2 - t_1},
    \end{displaymath}
    where $\tv (t, u_o,u_b)$ is as in~\eqref{eq:4} and
    $\mathcal{U} = \mathcal{U} (u_o, {u_b}_{|[0,t_1 \vee t_2]})$, with
    the notation~\eqref{eq:19}.
  \item Total variation estimate: for $t {\in} [0,T]$, with the
    notation~\eqref{eq:4},
    $\tv\left(u (t) \right)\! \leq\! \tv (t, u_o, u_b)$.
  \item $\L1$--Lipschitz continuity on initial and boundary data: if
    $\tilde u_o \in (\L1 \cap \BV) (\reali_+; \reali)$,
    $\tilde u_b \in (\L1 \cap \BV) ([0,T]; \reali)$ and $\tilde u$ is
    the corresponding solution to~\eqref{eq:1part}, for all
    $t \in [0,T]$,
    \begin{displaymath}
      \norma{u (t)- \tilde u (t)}_{\L1 (\reali_+;\reali)}
      \leq
      \norma{u_o- \tilde u_o}_{\L1 (\reali_+;\reali)}
      +
      \norma{v}_{\L\infty ([0,t];\reali)}
      \norma{g'}_{\L\infty (\mathcal{U};\reali)}
      \norma{u_b- \tilde u_b}_{\L1 ([0,t];\reali)},
    \end{displaymath}
    where
    $\mathcal{U} = \mathcal{U} ({u_b}_{|[0,t]}, {\tilde u}_{b|[0,t]})$
    as in~\eqref{eq:19}.
  \item $\L1$--stability with respect to $v$ and $g$: if
    $\tilde v \in \L\infty ([0,T] ; [v_{\min}, + \infty[)$,
    $\tilde g \in \C1 (\reali;\reali)$ and $\tilde u$ is the
    corresponding solution to~\eqref{eq:1part}, for all $t \in [0,T]$,
    \begin{displaymath}
      \norma{u (t)-\tilde u (t)}_{\L1 (\reali_+;\reali)}
      \leq
      \tv (t, u_o, u_b)
      \left(
        A \, t \, \norma{g' - \tilde g'}_{\L\infty (\mathcal{U};\reali)}
        +
        B  \, \norma{v-\tilde v}_{\L1 ([0,t];\reali)}
      \right),
    \end{displaymath}
    where $\tv (t, u_o, u_b)$ is defined in~\eqref{eq:4},
    $\mathcal{U} = \mathcal{U} (u_o; {u_b}_{\vert [0,t]})$ is as
    in~\eqref{eq:19} and we use the notation~\eqref{eq:15}.
  \end{enumerate}
\end{theorem}

\noindent The proof is deferred to \S~\ref{subs:disc}.

\section{The Case of a Segment}
\label{sec:seg}

In this Section we focus on the IBVP~\eqref{eq:1sgPart} where $x$
varies in a segment. We follow the same structure of
Section~\ref{sec:Main} and provide all statements in details, though
omitting the proofs, since they are entirely analogous to those
presented in \S~\ref{sec:vLip} and \S~\ref{subs:disc} for the case of
the half line, now relying on~\cite[\S~3.1]{IBVP1D}.

Consider first the more general IBVP on a segment
\begin{equation}
  \label{eq:1sg}
  \left\{
    \begin{array}{l@{\,}c@{\,}lr@{\,}c@{\,}l}
      \multicolumn{3}{l}{%
      \partial_t u + \partial_x f (t,u) = 0\qquad\ }
      &  (t,x)
      & \in
      & [0,T]  \times [0,L]
      \\
      u (0,x)
      & =
      & u_o (x)
      & x
      & \in
      & [0,L]
      \\
      u (t, 0)
      & =
      & u_{b,1} (t)
      &  t
      & \in & [0,T]
      \\
      u (t,L)
      & =
      & u_{b,2} (t)
      & t
      & \in
      & [0,T].
    \end{array}
  \right.
\end{equation}
The definition of solution to~\eqref{eq:1sg} is analogous to
Definition~\ref{def:solhl}. Here we have one more term, due to the
boundary $x = L$.

\begin{definition}
  \label{def:solsg}
  A \emph{solution} to the IBVP~\eqref{eq:1sg} is a map
  $u \in \L\infty ([0,T] \times [0,L]; \reali)$ such that for any
  $k \in \reali$ and for any test function
  $\phi \in \Cc1 (\reali \times \reali; \reali_+)$
  \begin{align*}
    & \int_0^T \!\!\! \int_0^L
      \left\{
      \left(u (t,x) - k\right)^\pm  \partial_t\phi (t,x)
      +
      \sgn{}^\pm \!\left(u (t,x) - k \right)\!
      \left(f \left(t, u (t,x)\right) - f (t,k)\right)  \partial_x \phi (t,x)
      \right\}
      \d{x} \d{t}
    \\
    & + \int_0^L \left(u_o (x) - k\right)^\pm  \phi (0, x) \,
      \d{x}
      - \int_0^L \left(u (T,x) - k\right)^\pm \phi (T, x) \,
      \d{x}
    \\
    & + \norma{\partial_u f}_{\L\infty ([0,T] \times \mathcal{U}; \reali)}
      \int_0^T \left\{
      \left(u_{b,1} (t) - k\right)^\pm \phi (t,0)
      +
      \left(u_{b,2} (t) - k\right)^\pm \phi (t,L)
      \right\} \d{t}
      \geq  0,
  \end{align*}
  where
  $\mathcal{U} = \mathcal{U} (u_o, {u_{b,1}}_{|[0,T]},
  {u_{b,2}}_{|[0,T]})$ as in~\eqref{eq:19}.
\end{definition}
\noindent Throughout, we denote $\ub = \left(u_{b,1}, u_{b,2}\right)$,
so that $\mathcal{U}=\mathcal{U} (u_o, {\ub}_{|[0,T]})$.

\subsection{Lipschitz continuous $\boldsymbol{v}$}
\label{sec:lipschitzsg}

We focus on problem~\eqref{eq:1sgPart} under the assumption
$v \in \C{0,1} ([0,T]; [v_{\min}, +\infty[)$.

The following Proposition extends~\cite[Proposition~3.2]{IBVP1D} to
the present setting of less regularity assumptions on the flux. It is
the analogous to Proposition~\ref{prop:nonso}, but the half line is
here replaced by a segment.


\begin{proposition}
  \label{prop:nonsosg}
  Let $v \in \C{0,1} ([0,T]; [v_{\min}, +\infty[)$. Fix
  $g \in \C1 (\reali; \reali)$. Let
  $u_o, \tilde u_o \in (\L1 \cap \BV) ([0,L]; \reali)$,
  $\ub, \boldsymbol{\tilde u_b} \in (\L1 \cap \BV) ([0,T];
  \reali^2)$. Assume that the IBVPs
  \begin{displaymath}
    \left\{
      \begin{array}{l}
        \partial_t u + \partial_x \left(v(t) \, g(u)\right) = 0
        \\
        u (0,x)
        =
        u_o (x)
        \\
        \left(u (t, 0), u (t,L)\right)
        =
        \ub (t)
      \end{array}
    \right.
    \quad \mbox{ and } \quad
    \left\{
      \begin{array}{l}
        \partial_t \tilde u + \partial_x \left(v(t) \, g(\tilde u)\right) = 0
        \\
        \tilde u (0,x)
        =
        \tilde u_o (x)
        \\
        \left(\tilde u (t, 0), \tilde u (t,L) \right)
        =
        \boldsymbol{\tilde u_b} (t)
      \end{array}
    \right.
  \end{displaymath}
  admit solutions
  $u,\tilde u \in \L\infty ([0,T] \times [0,L]; \reali)$ in the sense
  of Definition~\ref{def:solsg}, such that $u$ and $\tilde u$ both
  admit a trace for $x \to 0+$ and at $x \to L-$ for
  a.e.~$t \in [0,T]$. Then, for all $t \in [0,T]$,
  \begin{displaymath}
    \norma{u (t) \! - \! \tilde u (t)}_{\L1 ([0,L]; \reali)}
    \leq
    \norma{u_o  \! - \! \tilde u_o}_{\L1 ([0,L]; \reali)}
    +
    \norma{v}_{\L\infty ([0,t];\reali)}
    \norma{g'}_{\L\infty (\mathcal{U};\reali)}\!
    \sum_{i=1}^2\norma{u_{b.i} - \tilde u_{b,i}}_{\L1 ([0,t]; \reali)},
  \end{displaymath}
  where
  $\mathcal{U} = \mathcal{U} ( {\ub}_{|[0,t]}, \boldsymbol{\tilde
    u_b}_{|[0,t]} )$ is as in~\eqref{eq:19}.
\end{proposition}

\noindent The analogous to Proposition~\ref{prop:2part} reads as
follows.
\begin{proposition}
  \label{prop:2sgPart}
  Let $v \in \C{0,1} ([0,T]; [v_{\min}, +\infty[)$ and
  $g \in \Wloc1\infty (\reali; \reali)$. Fix
  $u_o \in (\L1 \cap \BV) ([0,L]; \reali)$ and
  $\ub \in (\L1 \cap \BV) ([0,T]; \reali^2)$. Then
  problem~\eqref{eq:1sgPart} admits a solution $u$ in the sense of
  Definition~\ref{def:solsg} with the properties:
  \begin{enumerate}
  \item Range of $u$: with the notation in~\eqref{eq:19},
    $u (t,x) \in \mathcal{U} (u_o, {\ub}_{|[0,t]})$ for
    a.e.~$(t,x) \in [0,T] \times [0,L]$. Hence, for all $t\in [0,T]$,
    \begin{displaymath}
      \norma{u (t)}_{\L\infty  ([0,L]; \reali)} \leq
      \max \left\{ \norma{u_o}_{\L\infty ([0,L]; \reali)}, \,
        \norma{u_{b,1}}_{\L\infty ([0,t];\reali)}, \,
        \norma{u_{b,2}}_{\L\infty ([0,t];\reali)}
      \right\}.
    \end{displaymath}
  \item $u$ is $\L1$--Lipschitz continuous in time: for all
    $t_1,t_2 \in [0,T]$,
    \begin{displaymath}
      \norma{u (t_1) - u (t_2)}_{\L1 ([0,L]; \reali)}
      \leq \tv (t_1\vee t_2, u_o, \ub) \;
      \norma{v}_{\L\infty ([0,t_1 \vee t_2];\reali)} \,
      \norma{g'}_{\L\infty (\mathcal{U};\reali)} \,
      \modulo{t_2 - t_1},
    \end{displaymath}
    where
    \begin{equation}
      \label{eq:4sg}
      \tv (t, u_o, \ub) {=}
      \tv (u_o) + \tv (\ub; [0,t])
      + \modulo{u_{b,1} (0+) - u_o (0+)} +\modulo{u_{b,2} (0+) - u_o (L-)},
    \end{equation}
    and $\mathcal{U} = \mathcal{U} (u_o, {\ub}_{|[0,t_1 \vee t_2]})$,
    with the notation~\eqref{eq:19}.
  \item Total variation estimate: for all $t \in [0,T]$,
    $\tv\left(u (t) \right) \leq \tv (t, u_o, \ub)$, with the
    notation~\eqref{eq:4sg}.
  \end{enumerate}
\end{proposition}

\noindent We conclude this paragraph with the following Theorem,
stating the stability with respect to the flux of the solution
to~\eqref{eq:1sgPart}, under the hypothesis
$v \in \C{0,1} ([0,T]; [v_{\min}, +\infty[)$.
\begin{theorem}
  \label{thm:3sgPart}
  Let $v, \tilde v \in \C{0,1} ([0,T]; [v_{\min}, +\infty[)$. Fix
  $g , \tilde g \in \C1 (\reali; \reali)$.  Let
  $u_o \in (\L1 \cap \BV) ([0,L]; \reali)$,
  $\ub \in (\L1 \cap \BV) ([0,T]; \reali^2)$. Call $u$ and $\tilde u$
  the solutions to the IBVP~\eqref{eq:1sgPart}, with flux $v\,g$ and
  $\tilde v \, \tilde g$ respectively. Then the following estimate
  holds: for all $t \in [0,T]$
  \begin{displaymath}
    \norma{u (t)-\tilde u (t)}_{\L1 ([0,L];\reali)}
    \leq
    \tv (t, u_o, \ub)
    \left(
      A \, t \, \norma{g' - \tilde g'}_{\L\infty (\mathcal{U};\reali)}
      +
      B  \, \norma{v-\tilde v}_{\L1 ([0,t];\reali)}
    \right),
  \end{displaymath}
  where $\tv (t, u_o, \ub)$ is defined in~\eqref{eq:4sg},
  $\mathcal{U} = \mathcal{U} (u_o; {\ub}_{\vert [0,t]})$ is as
  in~\eqref{eq:19}, $A$ and $B$ are defined in~\eqref{eq:15}.
\end{theorem}

\subsection{Discontinuous $\boldsymbol{v}$}
\label{sec:discontsg}

We now relax the regularity hypothesis on $v$, allowing also for
discontinuous functions. In particular, we consider
problem~\eqref{eq:1sgPart} with
$v \in \L\infty ([0,T] ; [v_{\min}, + \infty[)$.

\begin{theorem}\label{thm:mainsg}
  Let $v \in \L\infty ([0,T] ; [v_{\min}, + \infty[)$,
  $g \in \C1 (\reali;\reali)$,
  $u_o \in (\L1 \cap \BV) ([0,L]; \reali)$ and
  $\ub \in (\L1 \cap \BV) ([0,T]; \reali^2)$. Then
  problem~\eqref{eq:1sgPart} admits a solution $u$ in the sense of
  Definition~\ref{def:solsg} with the properties:
  \begin{enumerate}
  \item Range of $u$: with the notation in~\eqref{eq:19},
    $u (t,x) \in \mathcal{U} (u_o, {\ub}_{|[0,t]})$ for
    a.e.~$(t,x) \in [0,T] \times [0,L]$. Hence, for all $t\in [0,T]$,
    \begin{displaymath}
      \norma{u (t)}_{\L\infty ([0,L]; \reali)} \leq \max \left\{
        \norma{u_o}_{\L\infty ([0,L]; \reali)}, \,
        \norma{u_{b,1}}_{\L\infty ([0,t];\reali)}, \,
        \norma{u_{b,2}}_{\L\infty ([0,t];\reali)} \right\}.
    \end{displaymath}
  \item $u$ is $\L1$--Lipschitz continuous in time: for all
    $t_1,t_2 \in [0,T]$,
    \begin{displaymath}
      \norma{u (t_1) - u (t_2)}_{\L1 ([0,L]; \reali)}
      \leq \tv (t_1 \vee t_2, u_o, \ub) \;
      \norma{v}_{\L\infty ([0,t_1 \vee t_2];\reali)} \,
      \norma{g'}_{\L\infty (\mathcal{U};\reali)} \,
      \modulo{t_2 - t_1},
    \end{displaymath}
    where $\tv (t, u_o, \ub)$ is as in~\eqref{eq:4sg} and
    $\mathcal{U} = \mathcal{U} (u_o, {\ub}_{|[0,t_1 \vee t_2]})$, with
    the notation~\eqref{eq:19}.
  \item Total variation estimate: for all $t \in [0,T]$, with the
    notation~\eqref{eq:4},
    $\tv\left(u (t) \right) \leq \tv (t, u_o, \ub)$, where
    $\tv (t, u_o, \ub)$ is as in~\eqref{eq:4sg}.
  \item $\L1$--Lipschitz continuity on initial and boundary data: if
    $\tilde u_o \in (\L1 \cap \BV) ([0,L]; \reali)$,
    $\boldsymbol{\tilde u_b} \in (\L1 \cap \BV) ([0,T]; \reali^2)$ and
    $\tilde u$ is the corresponding solution to~\eqref{eq:1sgPart},
    for all $t \in [0,T]$,
    \begin{align*}
      \norma{u (t) \! - \! \tilde u (t)}_{\L1 ([0,L]; \reali)}
      \leq \
      &\norma{u_o  \! - \! \tilde u_o}_{\L1 ([0,L]; \reali)}
      \\
      &+
        \norma{v}_{\L\infty ([0,t];\reali)}
        \norma{g'}_{\L\infty (\mathcal{U};\reali)}
        \sum_{i=1}^2\norma{u_{b.i} - \tilde u_{b,i}}_{\L1 ([0,t]; \reali)},
    \end{align*}
    where
    $\mathcal{U} = \mathcal{U} ({\ub}_{|[0,t]}, \boldsymbol{\tilde
      u_b}_{|[0,t]})$ as in~\eqref{eq:19}.
  \item $\L1$--stability with respect to $v$ and $g$: if
    $\tilde v \in \L\infty ([0,T] ; [v_{\min}, + \infty[)$,
    $\tilde g \in \C1 (\reali;\reali)$ and $\tilde u$ is the
    corresponding solution to~\eqref{eq:1sgPart}, for all
    $t \in [0,T]$,
    \begin{displaymath}
      \norma{u (t)-\tilde u (t)}_{\L1 ([0,L];\reali)}
      \leq
      \tv (t, u_o, \ub)
      \left(
        A \, t \, \norma{g' - \tilde g'}_{\L\infty (\mathcal{U};\reali)}
        +
        B  \, \norma{v-\tilde v}_{\L1 ([0,t];\reali)}
      \right),
    \end{displaymath}
    where $\tv (t, u_o, \ub)$ is defined in~\eqref{eq:4sg},
    $\mathcal{U} = \mathcal{U} (u_o; {\ub}_{\vert [0,t]})$ is as
    in~\eqref{eq:19} and we use the notation~\eqref{eq:15}.
  \end{enumerate}
\end{theorem}

\section{Application to Vehicular Traffic}
\label{sec:example}

We consider below sample numerical integrations
of~\eqref{eq:1sgPart}. In the interior of the interval $[0,L]$, we
employ the standard Lax--Friedrichs
method~\cite[\S~4.6]{LeVequeBook2002}. Along the boundary, we
implement the Bardos, le Roux and
N\'ed\'elec~\cite{BardosLerouxNedelec} boundary condition in the
form~\cite[Proposition~2.3]{bordo}.

Consider a road segment of length $L = 250\,\mbox{m}$. Assume that at
the initial time the road is empty, that is to say, the initial datum
$u_o$ is equal to zero. At the entry of the road, a traffic light
remains green for $39\,\mbox{sec}$, while it displays red for
$27\, \mbox{sec}$ and, right at time $t=0$, the traffic light turns
green. Whenever the traffic light is green, the inflow is
$2000\,\mbox{cars}/\mbox{hour}$, see~\cite[\S~6.2]{RosiniBook} for
more details on assigning the inflow as boundary datum. At the end of
this road, a second traffic light regulates the outflow, being green
for $30 \, \mbox{sec}$, red for $45\,\mbox{sec}$, and first turning
red at time $t = 12\, \mbox{sec}$.

We describe the dynamics of traffic through the
Lighthill--Whitham~\cite{LighthillWhitham} and
Richards~\cite{Richards} model with time dependent maximal speed,
which amounts to~\eqref{eq:1sgPart} with
\begin{equation}
  \label{eq:7}
  g (u) = u  \left(1- \frac u R\right),
\end{equation}
$R$ being the maximal possible density, which is here considered to be
$200\, \mbox{cars}/\mbox{km}$. Concerning the time dependent (possibly
discontinuous) maximal speed $v (t)$, we let
\begin{equation}
  \label{eq:9}
  v (t) =
  \begin{cases}
    60 \,\mbox{km}/\mbox{hour} & \mbox{ if traffic light at $x=L$ is
      green,}
    \\
    V & \mbox{ if traffic light at $x=L$ is red,}
  \end{cases}
\end{equation}
where $V$ takes the values
\begin{align}
  \label{eq:8}
  40 \
  & \dfrac{\mbox{km}}{\mbox{hour}},
  &
    45 \
  & \dfrac{\mbox{km}}{\mbox{hour}},
  &
    50 \
  & \dfrac{\mbox{km}}{\mbox{hour}},
  &
    55 \
  & \dfrac{\mbox{km}}{\mbox{hour}},
  &
    60 \
  & \dfrac{\mbox{km}}{\mbox{hour}},
  &
    65 \
  & \dfrac{\mbox{km}}{\mbox{hour}},
  &
    70 \
  & \dfrac{\mbox{km}}{\mbox{hour}}.
\end{align}
These choices describe possible behaviours of drivers, reacting to the
traffic light in front of them either accelerating or slowing down. To
allow for reasonable comparisons among the different solutions, we
keep throughout the same inflow through the traffic light at $x=0$ as
well as the same outflow through the traffic light at $x=L$. As a
consequence, the global travel time, i.e.~the time necessary to empty
the road, is the same in all integrations. Note that the chosen inflow
of $2000 \, \mbox{cars/hour}$ does not allow choices of the maximal
speed lower than $40 \,\mbox{km/hour}$.

Remark, that the analytic setting presented in~\S~\ref{sec:discontsg}
applies, in particular, to the (possibly) discontinuous
choice~\eqref{eq:9}. Therefore, the resulting
model~\eqref{eq:1sgPart}--\eqref{eq:7}--\eqref{eq:9} is well posed.

To measure the queues formed due to the traffic light at $x=L$,
introduce the functional
\begin{equation}
  \label{eq:J}
  J = \int_0^T \int_{L-\delta}^L \Psi \left(u (t,x)\right) \d{x}  \d{t},
  \qquad \mbox{ where } \quad
  \delta = 100 \, \mbox{m}
\end{equation}
\begin{minipage}{0.38\linewidth}
  \includegraphics[width=0.95\textwidth, trim=10 5 40 10,
  clip=true]{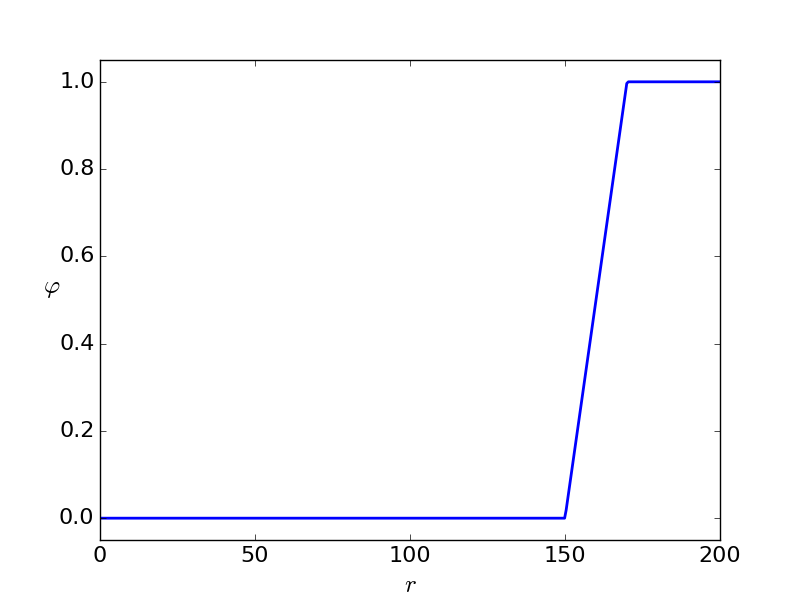} 
\end{minipage}%
\begin{minipage}{0.62\linewidth}
  \begin{equation}
    \label{eq:phi}
    \Psi (r) =
    \begin{cases}
      0 & \mbox{ if } r < 0.75 \, R,
      \\[3pt]
      \dfrac{10 \, r}{R} - 7.5 & \mbox{ if } 0.75 \,R \leq r \leq 0.85
      \, R,
      \\[3pt]
      1 & \mbox{ if } 0.85 \, R < r \leq R.
    \end{cases}
  \end{equation}
\end{minipage}\\
The function $\Psi$ weights $1$ wherever the traffic density is above
$85\%$ of the maximal density $R$, while it weights $0$ wherever the
vehicular density is lower than $75\%$ of $R$.

The values of $J$ resulting from the numerical integration
of~\eqref{eq:1sgPart}--\eqref{eq:7}--\eqref{eq:9}, in the different
cases~\eqref{eq:8}, are shown in Figure~\ref{fig:confronto}.
\begin{figure}[!h]
  \centering
  \includegraphics[width=0.6\textwidth, trim=20 0 45
  30,clip=true]{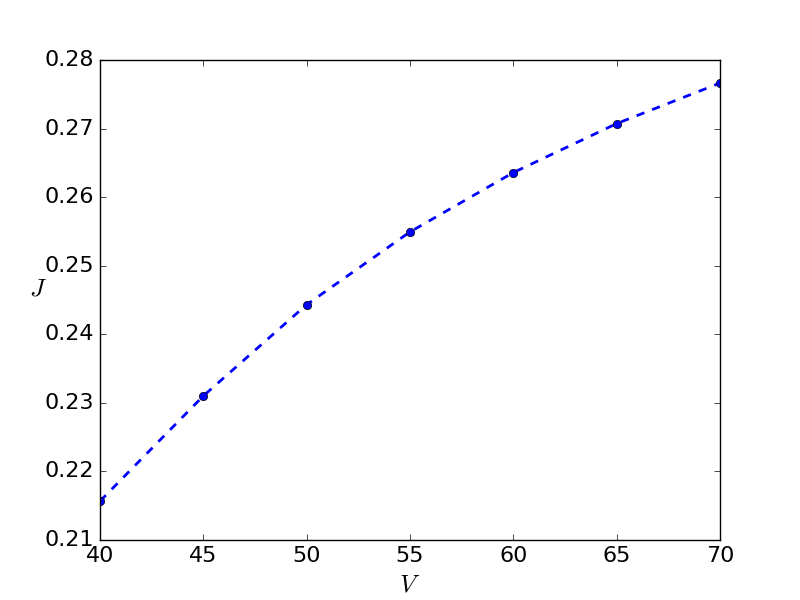}
  \caption{Values of $J$ as defined in~\eqref{eq:J}, resulting from
    the integration of~\eqref{eq:1sgPart}--\eqref{eq:7}--\eqref{eq:9},
    for $V$ varying as in~\eqref{eq:8}, with zero initial datum, while
    inflow and outflow are governed by traffic lights as specified in
    the text. }
  \label{fig:confronto}
\end{figure}
Note that the results in \S~\ref{sec:discontsg} apply to the present
setting and ensure the continuous dependence of $J$ on the parameter
$V$. Indeed, the map $V \to u$, where $u$
solves~\eqref{eq:1sgPart}--\eqref{eq:7}--\eqref{eq:9}, is continuous
with respect to the $\L1$ distance by Theorem~\ref{thm:mainsg}. The
continuity of the map $u \to J$ is immediate.

The best choice is clearly the one that corresponds to
$V = 40 \, \mbox{km/hour}$. The qualitative difference in the
evolution corresponding to the choices $V = 40 \, \mbox{km/hour}$ and
$V = 70 \, \mbox{km/hour}$ is displayed in Figure~\ref{fig:green}.
\begin{figure}[!h]
  \centering
  \includegraphics[width=0.5\textwidth, trim = 50 0 50 30, clip =
  true]{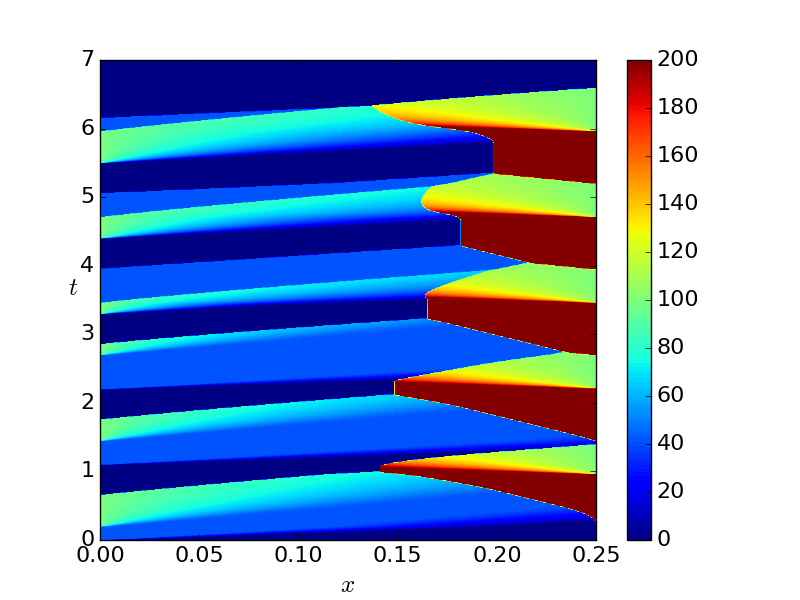}%
  \hfil
  \includegraphics[width=0.5\textwidth, trim = 50 0 50 30, clip =
  true]{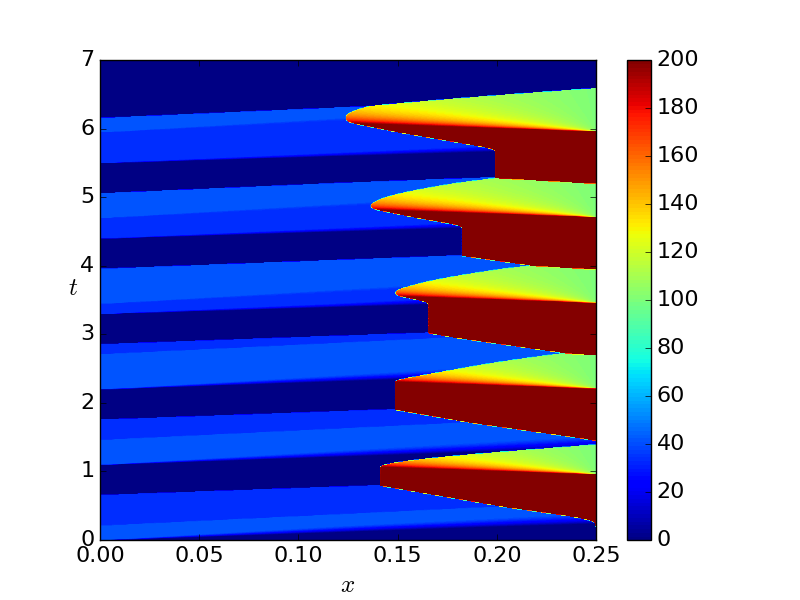}%
  \caption{Numerical integration
    of~\eqref{eq:1sgPart}--\eqref{eq:7}--\eqref{eq:9}, left with
    $V = 40 \, \mbox{km/hour}$ and, right, with
    $V = 70\,\mbox{km/hour}$. The initial datum is zero, inflow and
    outflow are both governed by traffic lights as specified in the
    text.\label{fig:green}}
\end{figure}
These graphs confirm that the speed reductions allows to reduce the
queue lengths.

\section{Technical Proofs}
\label{sec:TP}

\subsection{Preliminary Results}

We recall below the Lipschitz continuous dependence of the solution
to~\eqref{eq:1} on initial and boundary data.

\begin{proposition}[{\cite[Proposition~2.2]{IBVP1D}}]
  \label{prop:Elena}
  Let $f \in \C1([0,T] \times \reali; \reali)$ be such that the map
  $\{u \mapsto \partial_t f (t,u)\} \in \Wloc1\infty (\reali; \reali)$
  for all $t \in [0,T]$,
  $u_o, \tilde u_o \in (\L1 \cap \BV) (\reali_+; \reali)$ and
  $u_b, \tilde u_b \in (\L1 \cap \BV) ([0,T]; \reali)$. Assume the
  problems
  \begin{displaymath}
    \!\!\!
    \left\{
      \begin{array}{@{\,}l@{\,}c@{\,}lr@{\,}c@{\,}l}
        \multicolumn{3}{l}{%
        \partial_t u + \partial_x f (t,u) = 0\ }
        & (t,x)
        & \in
        & [0,T] \times \reali_+
        \\
        u (0,x)
        & =
        & u_o (x)
        & x
        & \in
        & \reali_+
        \\
        u (t, 0)
        & =
        & u_b (t)
        & t
        & \in
        & [0,T]
      \end{array}
    \right.
    \mbox{ and }
    \left\{
      \begin{array}{@{\,}l@{\,}c@{\,}lr@{\,}c@{\,}l}
        \multicolumn{3}{l}{%
        \partial_t \tilde u + \partial_x f (t,\tilde u) = 0\ }
        & (t,x)
        & \in
        & [0,T] \times \reali_+
        \\
        \tilde u (0,x)
        & =
        & \tilde u_o (x)
        & x
        & \in
        & \reali_+
        \\
        \tilde u (t, 0)
        & =
        & \tilde u_b (t)
        & t
        & \in
        & [0,T]
      \end{array}
    \right.
    \!\!
  \end{displaymath}
  admit solutions
  $u,\tilde u \in \L\infty ([0,T] \times \reali_+; \reali)$ in the
  sense of Definition~\ref{def:solhl}, such that $u$ and $\tilde u$
  both admit a trace for $x \to 0+$ for a.e.~$t \in [0,T]$. Then, for
  all $t \in [0,T]$,
  \begin{displaymath}
    \norma{u (t) - \tilde u (t)}_{\L1 (\reali_+; \reali)}
    \leq
    \norma{u_o - \tilde u_o}_{\L1 (\reali_+; \reali)}
    +
    \norma{\partial_u f}_{\L\infty ([0,t] \times \mathcal{U};\reali)} \;
    \norma{u_b - \tilde u_b}_{\L1 ([0,t]; \reali)}
  \end{displaymath}
  where
  $\mathcal{U} = \mathcal{U} ({u_b}_{|[0,t]}, \, {\tilde
    u}_{b|[0,t]})$ is as in~\eqref{eq:19}.
\end{proposition}

\noindent Remark that Proposition~\ref{prop:Elena} also ensures the
uniqueness of the solution to~\eqref{eq:1} in the sense of
Definition~\ref{def:solhl}, as soon as a solution exists.

Focus now on the particular case of the autonomous IBVP on the half
line:
\begin{equation}
  \label{eq:1aut}
  \left\{
    \begin{array}{@{\,}l@{\,}c@{\,}lr@{\,}c@{\,}l}
      \multicolumn{3}{l}{%
      \partial_t u + \partial_x g (u) = 0\qquad\ }
      & (t,x)
      & \in
      & [0,T] \times \reali_+
      \\
      u (0,x)
      & =
      & u_o (x)
      & x
      & \in
      & \reali_+
      \\
      u (t, 0)
      & =
      & u_b (t)
      & t
      & \in
      & [0,T] \,.
    \end{array}
  \right.
\end{equation}
As a definition of solution to~\eqref{eq:1aut}, we consider
Definition~\ref{def:solhl} discarding the explicit dependence of the
flux on time $t$.

\begin{proposition}[{\cite[Proposition~2.3]{IBVP1D}}]
  \label{prop:2}
  Let $g \in \Wloc{1}{\infty}(\reali; \reali)$,
  $u_o \in (\L1 \cap \BV) (\reali_+; \reali)$ and
  $u_b \in (\L1 \cap \BV) ([0,T]; \reali)$. Then
  problem~\eqref{eq:1aut} admits a solution $u$ in the sense of
  Definition~\ref{def:solhl}, with the properties:
  \begin{enumerate}
  \item If $u_o$ and $u_b$ are piecewise constant, then for $t$ small,
    the map $t \to u (t)$ coincides with the gluing of Lax solutions
    to Riemann problems at the points of jumps of $u_o$ and at $x=0$.
  \item Range of $u$: with the notation in~\eqref{eq:19},
    $u (t,x) \in \mathcal{U} (u_o, {u_b}_{|[0,t]})$ for
    a.e.~$(t,x) \in [0,T] \times \reali_+$. Hence, for all
    $t\in [0,T]$,
    $ \norma{u (t)}_{\L\infty (\reali_+; \reali)} \leq \max \left\{
      \norma{u_o}_{\L\infty (\reali_+; \reali)}, \,
      \norma{u_b}_{\L\infty ([0,t];\reali)} \right\}$.
  \item $u$ is $\L1$--Lipschitz continuous in time: for all
    $t_1,t_2 \in [0,T]$,
    \begin{align*}
      \nonumber
      \norma{u (t_1) - u (t_2)}_{\L1 (\reali_+; \reali)}
      \leq \
      & \tv (t_1\vee t_2, u_o,u_b) \;
        \norma{g'}_{\L\infty (\mathcal{U};\reali)} \;
        \modulo{t_2 - t_1},
    \end{align*}
    with the notation~\eqref{eq:4} and
    $\mathcal{U} = \mathcal{U} (u_o, {u_b}_{|[0,t_1 \vee t_2]})$,
    according to~\eqref{eq:19}.
  \item Total variation estimate: for $t {\in} [0,T]$, with the
    notation~\eqref{eq:4},
    $\tv\left(u (t) \right) {\leq} \tv (t, u_o, u_b)$.
  \end{enumerate}
\end{proposition}


\begin{theorem}[{\cite[Theorem~2.4]{IBVP1D}}]
  \label{thm:3}
  Let $g,\tilde g \in \C1 (\reali; \reali)$,
  $u_o \in (\L1 \cap \BV) (\reali_+; \reali)$ and
  $u_b \in (\L1 \cap \BV) ([0,T]; \reali)$. Call $u$ and $\tilde u$
  the solutions to the problems
  \begin{displaymath}
    \!\!\!
    \left\{
      \begin{array}{l@{\,}c@{\,}lr@{\,}c@{\,}l}
        \multicolumn{3}{l}{%
        \partial_t u + \partial_x g (u) = 0\ }
        & (t,x)
        & \in
        & [0,T] \times \reali_+
        \\
        u (0,x)
        & =
        & u_o (x)
        & x
        & \in
        & \reali_+
        \\
        u (t, 0)
        & =
        & u_b (t)
        & t
        & \in
        & [0,T]
      \end{array}
    \right.
    \mbox{ and }
    \left\{
      \begin{array}{l@{\,}c@{\,}lr@{\,}c@{\,}l}
        \multicolumn{3}{l}{%
        \partial_t \tilde u + \partial_x \tilde g (\tilde u) = 0\ }
        & (t,x)
        & \in
        & [0,T] \times \reali_+
        \\
        \tilde u (0,x)
        & =
        &  u_o (x)
        & x
        & \in
        & \reali_+
        \\
        \tilde u (t, 0)
        & =
        &  u_b (t)
        & t
        & \in
        & [0,T]
      \end{array}
    \right.
    \!\!
  \end{displaymath}
  constructed in Proposition~\ref{prop:2}. Then, with
  $\mathcal{U} = \mathcal{U} (u_o, {u_b}_{|[0,t]})$ as
  in~\eqref{eq:19}, for all $t \in [0,T]$,
  \begin{eqnarray*}
    \norma{u (t) - \tilde u (t)}_{\L1 (\reali_+; \reali)}
    & \leq
    & \max
      \left\{
      1, G
      \right\} \,
      \tv (t, u_o, u_b) \,
      \norma{g' - \tilde g'}_{\L\infty (\mathcal{U};\reali)} \,
      t ,
  \end{eqnarray*}
  with $\tv (t, u_o, u_b)$ as in~\eqref{eq:4} and
  $G = \min\left\{ \norma{g'}_{\L\infty (\mathcal{U}; \reali)},
    \norma{\tilde g'}_{\L\infty (\mathcal{U}; \reali)} \right\}$.
\end{theorem}

\subsection{Proofs Related to the Case $v$ Lipschitz Continuous}
\label{sec:vLip}

Let $v \in \C{0,1} ([0,T]; [v_{\min}, +\infty[)$ and let
$\hat T = \int_0^T v (\tau) \d\tau$. Define the bijective map
$\Gamma \colon [0,\hat T] \to [0,T]$ through its inverse
\begin{equation}
  \label{eq:gammaexpl}
  \Gamma^{-1}(t) = \int_0^t v (s)\d{s}
\end{equation}
so that
\begin{equation}
  \label{eq:Gamma}
  \Gamma'(\tau) \; v\left( \Gamma(\tau)\right) = 1
  \, , \qquad
  (\Gamma^{-1})' (t) = v(t)
  \,,\quad
  \Gamma (0) = 0
  \quad \mbox{ and } \quad
  \Gamma (\hat T) = T\,.
\end{equation}
We now establish the equivalence between the non autonomous
problem~\eqref{eq:1part} and an autonomous problem of
type~\eqref{eq:1aut}. Throughout, by \emph{solution} we mean
\emph{solution in the sense of Definition~\ref{def:solhl}}.

\begin{lemma}
  \label{lem:Gamma}
  Let $v \in \C{0,1} ([0,T]; [v_{\min}, +\infty[)$ and
  $g \in \C1 (\reali; \reali)$.  Define $\Gamma$ as
  in~\eqref{eq:gammaexpl}. Then, if $u$ solves~\eqref{eq:1part} the
  map
  \begin{equation}
    \label{eq:5}
    w (\tau,x) = u \left( \Gamma (\tau), x \right)
  \end{equation}
  solves
  \begin{equation}
    \label{eq:1w}
    \left\{
      \begin{array}{r@{\;}c@{\;}l@{\qquad}r@{\,}c@{\,}l}
        \multicolumn{3}{l}{\partial_\tau w  + \partial_x g(w) = 0\qquad}
        & (\tau,x)
        & \in
        & [0,\hat T] \times \reali_+
        \\
        w (0,x)
        & =
        & u_o (x)
        & x
        & \in
        & \reali_+
        \\
        w (\tau, 0)
        & =
        & w_b (\tau)
        & \tau
        & \in
        & [0,\hat T],
      \end{array}
    \right.
  \end{equation}
  where $w_b (\tau) = u_b \left(\Gamma (\tau) \right)$.  Conversely,
  if $w$ solves~\eqref{eq:1w}, then the map
  \begin{equation}
    \label{eq:uw}
    u(t, x) = w (\Gamma^{-1} (t), x)
  \end{equation}
  solves~\eqref{eq:1part}, where
  $u_b (t) = w_b \left(\Gamma^{-1} (t)\right)$.
\end{lemma}

\begin{proof}
  For any $\phi \in \Cc1(\reali \times \reali; \reali_+)$ and for any
  $k \in \reali$, compute the quantity
  \begin{align*}
    \mathcal{A} = \
    & \int_0^{\hat T}\int_{\reali_+} \left(w(\tau,x)-k\right)^\pm
      \partial_\tau \phi (\tau,x) \d{x} \d{\tau}
    \\
    &  + \int_0^{\hat T}\int_{\reali_+} \sgn{}^\pm \left(w(\tau,x)-k\right) \,
      \left( g \left( w(\tau,x)\right) - g(k) \right)  \,
      \partial_x \phi (\tau,x) \d{x} \d{\tau}
    \\
    & + \int_0^{\hat T} \sgn{}^\pm \left(u_b\left(\Gamma(\tau)\right)-k\right) \,
      \left( g \left( w(\tau,0+)\right) - g(k) \right)  \,
      \phi (\tau,0+)  \d{\tau}
    \\
    & + \int_{\reali_+}\left(u_o(x) -k\right)^\pm \phi(0,x) \d{x}
      - \int_{\reali_+} \left(w(\hat T,x) -k\right)^\pm \phi(\hat T,x) \d{x}
    \\
    = \
    & \int_0^{\hat T}\int_{\reali_+} \left(u\left(\Gamma(\tau),x\right)-k\right)^\pm
      \partial_\tau \phi (\tau,x) \d{x} \d{\tau}
    \\
    & + \int_0^{\hat T}\int_{\reali_+}
      \sgn{}^\pm\left(u\left(\Gamma(\tau),x\right)-k\right)
      \left( g \left( u(\Gamma(\tau),x) \right) - g(k) \right)  \,
      \partial_x \phi (\tau,x) \d{x} \d{\tau}
    \\
    & + \int_0^{\hat T} \sgn{}^\pm \left(u_b\left(\Gamma(\tau)\right)-k\right) \,
      \left( g \left( u\left(\Gamma(\tau),0+\right)\right) - g(k) \right)  \,
      \phi (\tau,0+)  \d{\tau}
    \\
    & + \int_{\reali_+} \left(u_o(x) -k\right)^\pm \phi(0,x) \d{x}
      - \int_{\reali_+} \left(u\left(\Gamma(\hat T),x\right) -k\right)^\pm
      \phi(\hat T,x) \d{x}.
  \end{align*}
  Use now the change of variable $\Gamma(\tau) = t$ and define
  $\psi(t,x) = \phi\left(\Gamma^{-1}(t),x\right)$. Clearly,
  $\psi \in \Cc1 (\reali\times\reali;\reali_+)$.  Using the
  properties~\eqref{eq:Gamma} of the function $\Gamma$, continue the
  computation:
  \begin{align*}
    \mathcal{A} =
    & \int_0^{T}\int_{\reali_+} \left(u\left(t,x\right)-k\right)^\pm \,
      \partial_t \psi \left(t,x\right) \d{x} \d{t}
    \\
    & + \int_0^{T}\int_{\reali_+} \sgn{}^\pm\left(u\left(t,x\right)-k\right) \,
      \left( g \left( u\left(t,x\right) \right) - g(k) \right) v(t) \,
      \partial_x \psi\left(t,x\right)\d{x} \d{t}
    \\
    & + \int_0^{T} \sgn{}^\pm \left(u_b\left(t\right)-k\right) \,
      \left( g \left( u\left(t,0+\right)\right) - g(k) \right) v(t) \,
      \psi\left(t,0+\right)  \d{t}
    \\
    & + \int_{\reali_+} \left(u_o(x) -k\right)^\pm \psi(0,x) \d{x}
      - \int_{\reali_+} \left( u\left(T,x\right) -k\right)^\pm
      \psi\left(T,x\right) \d{x}.
  \end{align*}
  Therefore, $u$ is a solution to~\eqref{eq:1part} on $[0,T]$ in the
  sense of Definition~\ref{def:solhl} if and only if $w$ is a solution
  to~\eqref{eq:1w} on $[0,\hat T]$ in the sense of
  Definition~\ref{def:solhl}.
\end{proof}

\begin{proofof}{Proposition~\ref{prop:nonso}}
  Call $w$ and $\tilde w$ the solutions to~\eqref{eq:1w} corresponding
  to $u$ and $\tilde u$, respectively, through the relation defined by
  Lemma~\ref{lem:Gamma}. For all $t \in [0,T]$, use~\eqref{eq:uw} and
  apply Proposition~\ref{prop:Elena} to $w$ and $\tilde w$:
  \begin{align*}
    \norma{u (t) - \tilde u (t)}_{\L1 (\reali_+;\reali)}
    = \
    & \norma{w (\Gamma^{-1} (t)) - \tilde w (\Gamma^{-1} (t))}_{\L1 (\reali_+;\reali)}
    \\
    \leq \
    & \norma{w_o -\tilde w_o}_{\L1 (\reali_+;\reali)} +
      \norma{g'}_{\L\infty (\mathcal{U};\reali)}
      \norma{w_b -\tilde w_b}_{\L1 ([0,\Gamma^{-1} (t)];\reali)}
    \\
    \leq \
    & \norma{u_o -\tilde u_o}_{\L1 (\reali_+;\reali)} +
      \norma{g'}_{\L\infty (\mathcal{U};\reali)}
      \norma{v}_{\L\infty ([0,t];\reali)}
      \norma{u_b -\tilde u_b}_{\L1 ([0,t];\reali)},
  \end{align*}
  where
  $\mathcal{U}=\mathcal{U}\left({w_b}_{|[0,\Gamma^{-1} (t)]}, {\tilde
      w}_{b|[0,\Gamma^{-1} (t)]}\right) =
  \mathcal{U}\left({u_b}_{|[0,t]}, {\tilde u}_{b|[0,t]}\right)$,
  and this concludes the proof.
\end{proofof}

\begin{proofof}{Proposition~\ref{prop:2part}} The existence of
  solutions follows from Lemma~\ref{lem:Gamma}, that allows to apply
  Proposition~\ref{prop:2}.

  In each of the steps below we exploit the correspondence $\Gamma$
  and the result of Lemma~\ref{lem:Gamma}.

  \paragraph{1. $\L\infty$--bound.} Thanks to~\eqref{eq:uw} and to
  Point~2.~in Proposition~\ref{prop:2}, for $t \in [0,T]$ we have
  \begin{displaymath}
    \norma{u (t)}_{\L\infty (\reali_+;\reali)} =
    \norma{w(\Gamma^{-1} (t))}_{\L\infty (\reali_+;\reali)}
    \leq
    \max\left\{\norma{u_o}_{\L\infty (\reali_+;\reali)},
      \norma{w_b}_{\L\infty ([0,\Gamma^{-1}(t)];\reali)} \right\}.
  \end{displaymath}
  Observe that
  \begin{displaymath}
    \norma{w_b}_{\L\infty ([0,\Gamma^{-1}(t)];\reali)}
    {=} \! \sup_{\tau \in [0,\Gamma^{-1}(t)]} \! \modulo{w_b (\tau)}
    = \! \sup_{\tau \in [0,\Gamma^{-1}(t)]} \! \modulo{u_b \left(\Gamma(\tau)\right)}
    = \sup_{s \in [0,t]}\modulo{u_b (t)}
    = \norma{u_b}_{\L\infty ([0,t];\reali)},
  \end{displaymath}
  hence $\displaystyle
  \norma{u (t)}_{\L\infty (\reali_+;\reali)} \leq
  \max\left\{\norma{u_o}_{\L\infty (\reali_+;\reali)},
    \norma{u_b}_{\L\infty ([0,t];\reali)} \right\}$.  Moreover,
  \begin{displaymath}
    u (t,x) = \ w(\Gamma^{-1} (t),x) \ \in \ \mathcal{U} (u_o,
    {w_b}_{|[0,\Gamma^{-1} (t)]}) =  \ \mathcal{U}(u_o, {u_b}_{|[0,t]}),
  \end{displaymath}
  for a.e~$(t,x) \in [0,T] \times \reali_+$, concluding the proof of
  Point~1.

  \paragraph{2. $\L1$--Lipschitz continuity in time.} Thanks
  to~\eqref{eq:uw} and to Point~3.~in Proposition~\ref{prop:2},
  \begin{align*}
    \norma{u (t_1) - u (t_2)}_{\L1 (\reali_+;\reali)}
    = \
    & \norma{w (\Gamma^{-1} (t_1)) - w(\Gamma^{-1} (t_2))}_{\L1 (\reali_+;\reali)}
    \\
    \leq \
    & \tv (\Gamma^{-1} (t_1) \vee \Gamma^{-1} (t_2), w_o, w_b) \,
      \norma{g'}_{\L\infty (\mathcal{U};\reali)} \,
      \modulo{\Gamma^{-1} (t_1) - \Gamma^{-1} (t_2)},
  \end{align*}
  with the notation in~\eqref{eq:4} and
  $\mathcal{U} = \mathcal{U} (w_o, {w_b}_{|[0; \Gamma^{-1} (t_1) \vee
    \Gamma^{-1} (t_2)]}) = \mathcal{U} (u_o, {u_b}_{|[0; t_1 \vee
    t_2]})$, according to~\eqref{eq:19}. Observe that:
  \begin{align*}
    \tv (w_b; [0, \Gamma^{-1} (t_1) \vee \Gamma^{-1} (t_2)]) = \
    & \tv (u_b; [0, t_1\vee t_2]) \,,
    \\
    \modulo{w_b(0+) - w_o(0+)} = \
    & \modulo{u_b(0+) - u_o(0+)} \,,
    \\
    \modulo{\Gamma^{-1} (t_1) - \Gamma^{-1} (t_2)} \leq \
    & \norma{(\Gamma^{-1})'}_{\L\infty ([0;t_1 \vee t_2]; \reali)}
      \modulo{t_1-t_2} = \norma{v}_{\L\infty ([0;t_1 \vee t_2]; \reali)} \modulo{t_1-t_2} \,,
  \end{align*}
  where we use~\eqref{eq:Gamma}. Therefore, referring also
  to~\eqref{eq:4},
  \begin{displaymath}
    \norma{u (t_1) - u (t_2)}_{\L1 (\reali_+;\reali)}
    \leq
    \tv (t_1 \vee t_2, u_o, u_b) \,
    \norma{v}_{\L\infty ([0;t_1 \vee t_2]; \reali)}\,
    \norma{g'}_{\L\infty (\mathcal{U};\reali)}  \modulo{t_1-t_2},
  \end{displaymath}
  proving Point~2.

  \paragraph{3. Total variation estimate.} For all $t \in [0,T]$,
  thanks to~\eqref{eq:uw} and to Point~4.~in Proposition~\ref{prop:2},
  using the notation~\eqref{eq:4}, we have
  \begin{displaymath}
    \tv \left(u (t)\right) =
    \tv\left(w (\Gamma^{-1} (t))\right)
    \leq
    \tv \left(\Gamma^{-1} (t), w_o, w_b\right)
    = \tv (t, u_o, u_b) ,
  \end{displaymath}
  proving Point~3.
\end{proofof}

\begin{lemma}
  \label{lem:AFP}
  Let $u \in \BV (\reali_+;\reali) $ and
  $\phi \colon \reali_+ \to \reali_+$ be measurable. Then, for all
  $y\in \reali_+$,
  \begin{displaymath}
    \int_0^y \modulo{u\left(x+\phi (x)\right) - u (x)}\d{x}
    \leq
    \norma{\phi}_{\L\infty ([0,y];\reali)}
    \,
    \tv(u;[0,y]).
  \end{displaymath}
\end{lemma}

\begin{proof}
  Approximate the function $\phi$ with a sequence of simple functions
  $\phi_k$, so that $\phi_k \to \phi$ pointwise a.e.~and
  $\phi_k\leq \phi$. In particular
  \begin{displaymath}
    \phi_k(x) = \sum_{n=1}^{N_k} \phi_n^k \, \caratt{[x_{n-1}^k,x_{n}^k[}(x)
  \end{displaymath}
  for suitable $N_k \in \naturali$ and
  $0=x_0^k<x_1^k<\ldots < x_{N_k-1}^k<x_{N_k}^k=y$. Then, for any
  $k \in \naturali$,
  \begin{align*}
    \int_0^y \modulo{u\left(x+\phi_k (x)\right)-u(x)}\d{x}
    \leq
    & \sum_{n=1}^{N_k} \int_{x_{n-1}^k}^{x_n^k}
      \modulo{u\left(x+\phi_n^k\right) - u (x)}\d{x}
    \\
    \leq \
    & \sum_{n=1}^{N_k} \phi_n^k \, \tv\left(u;[x_{n-1}^k,x_n^k[\right)
      \qquad \mbox{[\small{by~\cite[Remark~3.25]{AmbrosioFuscoPallara}}]}
    \\
    \leq \
    & \norma{\phi_k}_{\L\infty ([0,y];\reali)} \, \tv(u; [0,y])
    \\
    \leq \
    & \norma{\phi}_{\L\infty ([0,y];\reali)} \, \tv(u; [0,y]).
  \end{align*}
  Since $\phi_k \to \phi$ pointwise a.e., we obtain the thesis.
\end{proof}

\begin{proofof}{Theorem~\ref{thm:3part}} For simplicity, we deal
  separately with the cases $v = \tilde v$ and $g = \tilde g$.

  \paragraph{Stability w.r.t.~$g$:} Assume first $v = \tilde v$.
  Then the autonomous problems corresponding to~\eqref{eq:quella}
  through $\Gamma$ as defined in~\eqref{eq:gammaexpl} are
  \begin{displaymath}
    \left\{
      \begin{array}{l}
        \partial_\tau w + \partial_x  g(w) = 0
        \\
        w (0,x) = u_o (x)
        \\
        w (\tau, 0) = u_b \left(\Gamma (\tau)\right)
      \end{array}
    \right. \quad \mbox{ and } \quad
    \left\{
      \begin{array}{l}
        \partial_\tau \tilde w + \partial_x  \tilde g(\tilde w) = 0
        \\
        \tilde w (0,x) = u_o (x)
        \\
        \tilde w (\tau, 0) = u_b \left(\Gamma (\tau)\right).
      \end{array}
    \right.
  \end{displaymath}
  For all $t \in [0,T]$, apply Lemma~\ref{lem:Gamma} and
  Theorem~\ref{thm:3} to obtain
  \begin{align*}
    & \norma{u (t)-\tilde u (t)}_{\L1 (\reali_+;\reali)}
    \\
    = \
    & \norma{
      w (\Gamma^{-1} (t)) -\tilde w (\Gamma^{-1} (t))
      }_{\L1 (\reali_+;\reali)}
    \\
    \leq \
    & \tv (t, u_o, u_b)
      \max \left\{
      1,
      \min\left\{
      \norma{g'}_{\L\infty (\mathcal{U};\reali)},
      \norma{\tilde g'}_{\L\infty (\mathcal{U};\reali)}
      \right\}
      \right\}
      \norma{g' - \tilde g'}_{\L\infty (\mathcal{U};\reali)}
      \, \Gamma^{-1} (t)
    \\
    \leq \
    & \tv (t, u_o, u_b)
      \max\left\{
      1,
      \min
      \left\{
      \norma{g'}_{\L\infty (\mathcal{U};\reali)},
      \norma{\tilde g'}_{\L\infty (\mathcal{U};\reali)}
      \right\}
      \right\}
      \norma{g' - \tilde g'}_{\L\infty (\mathcal{U};\reali)}
      \, \norma{v}_{\L\infty ([0,t];\reali)} \, t,
  \end{align*}
  with the notation~\eqref{eq:4} and with
  $\mathcal{U} = \mathcal{U} (u_o,{u_b}_{|[0,t]})$ as
  in~\eqref{eq:19}.

  \paragraph{Stability w.r.t.~$v$:} Assume now $g=\tilde g$. Let
  $\Gamma$ be as in~\eqref{eq:gammaexpl} and call $\tilde\Gamma$ the
  analogous function associated to $\tilde v$. Through $\Gamma$ and
  $\tilde\Gamma$ we apply Lemma~\ref{lem:Gamma} to the autonomous
  problems
  \begin{displaymath}
    \left\{
      \begin{array}{l}
        \partial_\tau w + \partial_x  g(w) = 0
        \\
        w (0,x) = u_o (x)
        \\
        w (\tau, 0) = w_b (\tau)
      \end{array}
    \right. \quad \mbox{ and } \quad
    \left\{
      \begin{array}{l}
        \partial_s \tilde w + \partial_x  g(\tilde w) = 0
        \\
        \tilde w (0,x) = u_o (x)
        \\
        \tilde w (s, 0) = \tilde w_b (s)
      \end{array}
    \right.
    \quad \mbox{where} \quad
    \begin{array}{@{}rcl@{}}
      w_b (\tau)
      & =
      & u_b \left(\Gamma (\tau)\right) ,
      \\
      \tilde w_b (s)
      & =
      & u_b \big(\tilde \Gamma (s)\big).
    \end{array}
  \end{displaymath}
  For all $t \in [0,T]$ we have
  \begin{align}
    \label{eq:boh}
    & \norma{u (t) - \tilde u (t)}_{\L1 (\reali_+;\reali)}
    \\
    \nonumber
    = \
    & \norma{w (\Gamma^{-1} (t)) -
      \tilde w (\tilde \Gamma^{-1}(t))}_{\L1 (\reali_+;\reali)}
    \\
    \label{eq:2}
    \leq \
    & \norma{w (\Gamma^{-1} (t)) - w (\tilde \Gamma^{-1}
      (t))}_{\L1 (\reali_+;\reali)} + \norma{w (\tilde \Gamma^{-1} (t)) -
      \tilde w (\tilde \Gamma^{-1} (t))}_{\L1 (\reali_+;\reali)}.
  \end{align}
  Consider the two terms in~\eqref{eq:2} separately. The first can be
  estimated using the Lipschitz continuity in time of the solutions to
  autonomous problems, i.e.~point~3.~in Proposition~\ref{prop:2}:
  \begin{displaymath}
    \norma{w (\Gamma^{-1} (t)) - w (\tilde \Gamma^{-1}(t))}_{\L1 (\reali_+;\reali)}
    \leq \tv (\Gamma^{-1} (t) \vee \tilde \Gamma^{-1} (t), u_o,w_b) \,
    \norma{g'}_{\L\infty (\mathcal{U};\reali)}
    \modulo{\Gamma^{-1} (t) - \tilde \Gamma^{-1} (t)},
  \end{displaymath}
  with the notation~\eqref{eq:4}, which leads to
  \begin{align*}
    \tv (\Gamma^{-1} (t) \vee \tilde \Gamma^{-1} (t), u_o,w_b) = \
    & \tv (u_o) + \tv (w_b; [0, \Gamma^{-1} (t) \vee \tilde \Gamma^{-1} (t)])
      + \modulo{w_b(0+) - u_o(0+)}
    \\
    = \
    & \tv (u_o) + \tv (u_b; [0, t \vee \Gamma(\tilde \Gamma^{-1} (t))])
      + \modulo{u_b(0+) - u_o(0+)}
    \\
    = \
    & \tv (t \vee \Gamma(\tilde \Gamma^{-1} (t)),u_o,u_b),
  \end{align*}
  and
  \begin{equation}
    \label{eq:3}
    \mathcal{U}=\mathcal{U}\left(u_o,{w_b}_{| [0, \Gamma^{-1} (t) \vee
        \tilde \Gamma^{-1} (t)]}\right) = \mathcal{U}\left(u_o,{u_b}_{|[0, t
        \vee \Gamma(\tilde \Gamma^{-1} (t))]}\right).
  \end{equation}
  Due to the definition~\eqref{eq:gammaexpl} of $\Gamma^{-1}$ and
  $\tilde \Gamma^{-1}$, we get
  \begin{equation}
    \label{eq:gammadif}
    \modulo{\Gamma^{-1} (t) - \tilde \Gamma^{-1} (t)}
    \leq \int_0^t \modulo{v (s) - \tilde v (s)}\d{s}
    =
    \norma{v - \tilde v}_{\L1 ([0,t];\reali)}.
  \end{equation}
  Therefore
  \begin{equation}
    \label{eq:2a}
    \norma{w (\Gamma^{-1} (t)) - w (\tilde \Gamma^{-1}(t))}_{\L1 (\reali_+;\reali)}
    \!\!\!
    \leq  \tv (t \vee \Gamma(\tilde \Gamma^{-1} (t)),u_o,u_b) \,
    \norma{g'}_{\L\infty (\mathcal{U};\reali)}
    \norma{v - \tilde v}_{\L1 ([0,t];\reali)}.
  \end{equation}
  The second term in~\eqref{eq:2} is the difference between solutions
  to autonomous IBVPs with different boundary data, computed at time
  $\tilde \Gamma^{-1} (t)$. By Proposition~\ref{prop:Elena}, with
  $\mathcal{U}$ as in~\eqref{eq:3}:
  \begin{align}
    \nonumber
    & \norma{w (\tilde \Gamma^{-1} (t)) -
      \tilde w (\tilde \Gamma^{-1} (t))}_{\L1 (\reali_+;\reali)}
    \\
    \nonumber
    \leq \
    & \norma{g'}_{\L\infty (\mathcal{U}; \reali)}
      \norma{w_b - \tilde w_b}_{\L1 ([0,\tilde \Gamma^{-1} (t)])}
    \\
    \nonumber
    = \
    & \norma{g'}_{\L\infty (\mathcal{U}; \reali)}
      \int_0^{\tilde \Gamma^{-1} (t)} \modulo{w_b (s) - \tilde w_b (s)}\d{s}
    \\
    \nonumber
    = \
    &\norma{g'}_{\L\infty (\mathcal{U}; \reali)}
      \int_0^{\tilde \Gamma^{-1} (t)}
      \modulo{u_b (\Gamma(s)) - u_b (\tilde \Gamma(s))}\d{s}
    \\
    \nonumber
    = \
    &  \norma{g'}_{\L\infty (\mathcal{U}; \reali)}
      \int_0^t
      \modulo{
      u_b(\Gamma (\tilde \Gamma^{-1} (\sigma))) - u_b (\sigma)} \,
      \tilde v (\sigma) \d\sigma
    \\
    \label{eq:10}
    \leq \
    & \norma{g'}_{\L\infty (\mathcal{U}; \reali)}
      \norma{\tilde v}_{\L\infty ([0,t];\reali)}
      \int_0^t
      \modulo{
      u_b(\Gamma (\tilde \Gamma^{-1} (\sigma))) - u_b (\sigma)} \,
      \d\sigma,
  \end{align}
  where we set $\tilde \Gamma (s)=\sigma$. Apply now
  Lemma~\ref{lem:AFP} to the integral term in~\eqref{eq:10}:
  \begin{align}
    \nonumber
    \int_0^t
    \modulo{
    u_b(\Gamma (\tilde \Gamma^{-1} (\sigma))) - u_b (\sigma)} \,
    \d\sigma
    \leq \
    & \tv\left(u_b;[0,t]\right)\,
      \sup_{\sigma \in [0,t]} \modulo{\Gamma (\tilde \Gamma^{-1} (\sigma)) - \sigma}
    \\
    \nonumber
    \leq \
    & \tv\left(u_b;[0,t]\right)\,
      \sup_{\sigma \in [0,t]}
      \modulo{\Gamma (\tilde \Gamma^{-1} (\sigma)) - \Gamma(\Gamma^{-1}(\sigma))}
    \\
    \nonumber
    \leq \
    & \tv\left(u_b;[0,t]\right)\,
      \mathrm{Lip}\, \Gamma
      \sup_{\sigma \in [0,t]}
      \modulo{\tilde \Gamma^{-1} (\sigma) - \Gamma^{-1} (\sigma)}
    \\ \nonumber
    \leq \
    & \tv\left(u_b;[0,t]\right)\,
      \mathrm{Lip}\, \Gamma
      \sup_{\sigma \in [0,t]}
      \left(\norma{v-\tilde v}_{\L1 ([0,\sigma];\reali)}\right)
    \\
    \label{eq:11}
    \leq \
    & \tv\left(u_b;[0,t]\right)\,
      \frac1{v_{\min}} \,
      \norma{v-\tilde v}_{\L1 ([0,t];\reali)}  \,,
  \end{align}
  where we exploited also~\eqref{eq:gammadif}. Using~\eqref{eq:11}
  in~\eqref{eq:10} yields
  \begin{align}
    \nonumber
    & \norma{w (\tilde \Gamma^{-1} (t)) -
      \tilde w (\tilde \Gamma^{-1} (t))}_{\L1 (\reali_+;\reali)}
    \\
    \label{eq:2b}
    \leq \
    & \norma{g'}_{\L\infty (\mathcal{U}; \reali)}
      \norma{\tilde v}_{\L\infty ([0,t];\reali)}
      \tv\left(u_b;[0,t]\right)\,
      \frac1{v_{\min}} \,
      \norma{v-\tilde v}_{\L1 ([0,t];\reali)} .
  \end{align}
  Insert now~\eqref{eq:2a} and~\eqref{eq:2b} in~\eqref{eq:2} to obtain
  \begin{align*}
    & \norma{u (t)-\tilde u (t)}_{\L1 (\reali_+;\reali)}
    \\
    \leq \
    & \left(1+ \norma{\tilde v}_{\L\infty ([0,t];\reali)}
      \frac1{v_{\min}} \right)\,
      \tv (t \vee \Gamma (\Gamma^{-1} (t)),u_o,u_b)
      \norma{g'}_{\L\infty (\mathcal{U}; \reali)}
      \norma{v-\tilde v}_{\L1 ([0,t];\reali)} .
  \end{align*}
  Observe that~\eqref{eq:boh} can be estimated also in the following
  way:
  \begin{displaymath}
    \norma{u (t) - \tilde u (t)}_{\L1 (\reali_+;\reali)}
    \leq
    \norma{w (\Gamma^{-1} (t)) - \tilde w (\Gamma^{-1} (t))}_{\L1 (\reali_+;\reali)}
    +
    \norma{\tilde w (\Gamma^{-1} (t)) -
      \tilde w (\tilde \Gamma^{-1} (t))}_{\L1 (\reali_+;\reali)},
  \end{displaymath}
  which yields a symmetric result.

  \smallskip

  We claim that the following inequality holds: for $t\in [0,T]$
  \begin{equation}
    \label{eq:etwas}
    \min\left\{
      \max\left\{
        t, \, \Gamma\left(\tilde \Gamma^{-1}(t)\right)
      \right\},
      \max\left\{
        t, \, \tilde \Gamma\left(\Gamma^{-1}(t)\right)
      \right\}
    \right\}
    \leq
    t.
  \end{equation}
  Indeed, if $t\geq \Gamma\left(\tilde \Gamma^{-1}(t)\right)$, then
  clearly $\Gamma^{-1}(t)\geq \tilde \Gamma^{-1}(t)$ and this implies
  $\tilde \Gamma\left(\Gamma^{-1}(t)\right)\geq t$.  Therefore, the
  left hand side in~\eqref{eq:etwas} now reads
  $\min\left\{ t, \tilde \Gamma\left(\Gamma^{-1}(t)\right) \right\} =
  t$.
  The case $t\leq \Gamma\left(\tilde \Gamma^{-1}(t)\right)$ leads to
  the same result, completing the proof of the claim.

  \smallskip

  Hence, exploiting~\eqref{eq:etwas} we obtain the following estimate
  for \eqref{eq:boh}:
  \begin{align*}
    \norma{u (t)-\tilde u (t)}_{\L1 (\reali_+;\reali)}
    \leq \
    & \left(1+
      \frac1{v_{\min}}
      \min\left\{\norma{v}_{\L\infty ([0,t];\reali)},
      \norma{\tilde v}_{\L\infty ([0,t];\reali)} \right\} \right)
    \\
    & \times \tv (t,u_o,u_b) \,
      \norma{g'}_{\L\infty (\mathcal{U}; \reali)} \,
      \norma{v-\tilde v}_{\L1 ([0,t];\reali)} ,
  \end{align*}
  where now
  $\mathcal{U}= \mathcal{U}\left( u_o; {u_b}_{|[0,t]} \right)$ thanks
  to~\eqref{eq:etwas}.
\end{proofof}

\subsection{Proof Related to the Case $v$ Discontinuous}
\label{subs:disc}

\begin{proofof}{Theorem~\ref{thm:main}}
  Let $\eta \in \Cc1(\reali;\reali_+)$ be a smooth mollifier, with
  $\spt \eta \subseteq [0,1]$ and
  $\norma{\eta}_{\L1 (\reali; \reali)} = 1$. For any $n \in \naturali$
  set $\eta_n (z) = n \, \eta(n \, z)$.  Define the sequence
  $v_n \in \C1 (\reali; \reali)$ as follows:
  \begin{displaymath}
    v_n = (\bar v * \eta_n)_{|[0,T]}
    \quad \mbox{ where } \quad
    \bar v (t) =
      \begin{cases}
        v (t) & t \in [0,T]
        \\
        v_{\min} & t \in \reali \setminus [0,T] \,.
      \end{cases}
  \end{displaymath}
  Clearly, $v_n$ converges to $v$ in $\L1 ([0,T]; \reali)$ and it is such
  that
  $v_{\min} \leq v_n (t) \leq \norma{v}_{\L\infty ([0,t];\reali)}$ for
  a.e.~$t\in [0,T]$.

  Obviously, $v_n \in \C{0,1} ([0,T]; [v_{\min}, +\infty[)$. By
  propositions~\ref{prop:nonso} and~\ref{prop:2part}, for any
  $n \in \naturali$, there exists a unique solution $u_n$ to the IBVP
  \begin{equation}
    \label{eq:un}
    \left\{
      \begin{array}{l@{\,}c@{\,}lr@{\,}c@{\,}l}
        \multicolumn{3}{l}{%
        \partial_t u_n + \partial_x \left(v_n(t) \, g(u_n)\right) = 0\qquad\ }
        & (t,x)
        & \in
        & [0,T] \times \reali_+
        \\
        u_n (0,x)
        & =
        & u_o (x)
        & x
        & \in
        & \reali_+
        \\
        u_n (t, 0)
        & =
        & u_b (t)
        & t
        & \in
        & [0,T].
      \end{array}
    \right.
  \end{equation}
  Moreover, by Theorem~\ref{thm:3part} and the properties of the
  sequence $v_n$, for any $n, m \in \naturali$ and for all
  $t \in [0,T]$, the following estimate holds:
  \begin{align*}
    \norma{u_n (t) - u_m (t)}_{\L1 (\reali_+;\reali)}
    \leq \
    & \tv (t, u_o, u_b)
      \left( 1 + \frac{\norma{v}_{\L\infty ([0,t];\reali)}}{v_{\min}}\right)
      \norma{g}_{\L\infty (\mathcal{U};\reali)} \,
      \norma{v_n - v_m}_{\L1 ([0,t];\reali)},
  \end{align*}
  where $\tv (t, u_o, u_b)$ is as in~\eqref{eq:4} and
  $\mathcal{U} = \mathcal{U} (u_o, u_{b|[0,t]})$ as
  in~\eqref{eq:19}. Therefore, $u_n$ is a Cauchy sequence in
  $\C0 ([0,T];\L1 (\reali_+;\reali))$, which is a complete metric
  space with the norm
  $\norma{u}_{\C0 ([0,T];\L1 (\reali_+;\reali))} = \sup_{t\in [0,T]}
  \norma{u (t)}_{\L1 (\reali_+; \reali)}$.
  Call $u$ the limit of the sequence $u_n$.

  The function $u$ has the following properties:

  \paragraph{1. $\L\infty$--bound.} By point~1.~in
  Proposition~\ref{prop:2part}, for all $t \in [0,T]$ we have the
  estimate
  $\norma{u_n (t)}_{\L\infty (\reali_+; \reali)} \leq \max \left\{
    \norma{u_o}_{\L\infty (\reali_+; \reali)}, \,
    \norma{u_b}_{\L\infty ([0,t];\reali)} \right\}$,
  uniformly in $n$. Hence, the same bound holds also on $u$, passing
  to the limit $n \to +\infty$, possibly on a subsequence. Since
  $u_n (t,x) \in \mathcal{U} (u_o, {u_b}_{|[0,t]})$ for
  a.e.~$(t,x)\in[0,T] \times \reali_+$ and for all $n$, also
  $u (t,x) \in \mathcal{U} (u_o, {u_b}_{|[0,t]})$.

  \paragraph{$\boldsymbol{u}$ is a solution.}  Since $u_n$ is a solution
  to~\eqref{eq:un}, for any $k\in \reali$ and for any test function
  $\phi \in \Cc1 (\reali \times \reali; \reali_+)$, it holds
  \begin{align}
    \label{eq:7a}
    & \int_0^T \!\!\! \int_{\reali_+}
      \left(u_n (t,x) - k\right)^\pm  \partial_t\phi (t,x)\d{x} \d{t}
    \\
    \label{eq:7b}
    & + \int_0^T \!\!\! \int_{\reali_+}
      \sgn{}^\pm \left(u_n (t,x) - k \right)
      \left(g \left(u_n (t,x)\right) - g (k)\right) v_n (t) \,
      \partial_x \phi (t,x)
      \d{x} \d{t}
    \\
    \label{eq:7c}
    & + \int_{\reali_+} \left(u_o (x) - k\right)^\pm  \phi (0, x) \,
      \d{x}
      - \int_{\reali_+} \left(u_n (T,x) - k\right)^\pm \phi (T, x) \,
      \d{x}
    \\\label{eq:7d}
    & + \norma{g'}_{\L\infty (\mathcal{U}; \reali)}
      \norma{v_n}_{\L\infty (([0,T];\reali)}
      \int_0^T \left(u_b (t) - k\right)^\pm \phi (t,0) \, \d{t}
      \geq  0,
  \end{align}
  with $\mathcal{U}=\mathcal{U} (u_o, u_{b|[0,T]})$ as
  in~\eqref{eq:19}. Compute the limit as $n \to + \infty$ of each line
  above separately. Concerning the first line we have
  \begin{align*}
    [\eqref{eq:7a}] \leq \
    &
      \int_0^T \!\!\! \int_{\reali_+}
      \left(u (t,x) - k\right)^\pm  \partial_t\phi (t,x)\d{x} \d{t}
    \\
    & +
      \int_0^T \!\!\! \int_{\reali_+}
      \left[ \left(u_n (t,x) - k\right)^\pm -
      \left(u (t,x) - k\right)^\pm  \right]
      \partial_t\phi (t,x)\d{x} \d{t}
    \\
    \leq \
    &
      \int_0^T \!\!\! \int_{\reali_+}
      \left(u (t,x) - k\right)^\pm  \partial_t\phi (t,x)\d{x} \d{t}
      +
      \int_0^T \!\!\! \int_{\reali_+}
      \modulo{u_n (t,x) -u (t,x)} \modulo{\partial_t \phi (t,x)}\d{x} \d{t}
  \end{align*}
  and the second term above tends to $0$ by the Dominated Convergence
  Theorem, so that in the limit we get
  \begin{displaymath}
    \lim_{n \to +\infty} [\eqref{eq:7a}] \leq
    \int_0^T \!\!\! \int_{\reali_+}
    \left(u (t,x) - k\right)^\pm  \partial_t\phi (t,x)\d{x} \d{t}.
  \end{displaymath}
  Pass now to~\eqref{eq:7b}. Compute
  \begin{align*}
    [\eqref{eq:7b}]
    \leq \
    &  \int_0^T \!\!\! \int_{\reali_+}
      \sgn{}^\pm \left(u (t,x) - k \right)
      \left(g \left(u (t,x)\right) - g (k)\right) v (t) \,
      \partial_x \phi (t,x)
      \d{x} \d{t}
    \\
    & + \int_0^T \!\!\! \int_{\reali_+}
      \sgn{}^\pm \left(u (t,x) - k \right)
      \left(g \left(u (t,x)\right) - g (k)\right)
      \left( v_n (t) - v (t) \right)\,
      \partial_x \phi (t,x)
      \d{x} \d{t}
    \\
    & + \int_0^T \!\!\! \int_{\reali_+}
      \Bigl[\sgn{}^\pm \left(u_n (t,x) - k \right)\!
      \left(g \left(u_n (t,x)\right) - g (k)\right)
    \\
    & \qquad\qquad\,\, -
      \sgn{}^\pm \left(u (t,x) - k \right)\!
      \left(g \left(u (t,x)\right) - g (k)\right)  \Bigr]  v_n (t)\,
      \partial_x \phi (t,x)
      \d{x} \d{t}.
  \end{align*}
  By the Dominated Convergence Theorem, the second line above vanishes
  in the limit $n \to + \infty$. The same happens both to the third
  and to the fourth line. Indeed, the map $g$ is in
  $\C1 (\reali;\reali)$, the map $u$ is bounded and hence
  $g_{|\mathcal{U}}$ is Lipschitz continuous. A slight extension
  of~\cite[Lemma~3]{Kruzkov} ensures that the map
  $u \to {\sgn}^\pm (u-k)\left(g (u) -g (k)\right)$ is Lipschitz
  continuous uniformly in $k$. Then:
  \begin{align*}
    & \int_0^T \!\!\! \int_{\reali_+}
      \Bigl[\sgn{}^\pm \left(u_n (t,x) - k \right)\!
      \left(g \left(u_n (t,x)\right) - g (k)\right)
    \\
    & \qquad\qquad\,\, -
      \sgn{}^\pm \left(u (t,x) - k \right)\!
      \left(g \left(u (t,x)\right) - g (k)\right)  \Bigr]  v_n (t)\,
      \partial_x \phi (t,x)
      \d{x} \d{t}
    \\
    \leq \
    & \int_0^T \!\!\! \int_{\reali_+}
      \norma{g'}_{\L\infty (\mathcal{U};\reali)} \,
      \modulo{u_n (t,x) - u (t,x)}  \, v_n (t) \,
      \modulo{\partial_x \phi (t,x)} \d{x} \d{t},
  \end{align*}
  which clearly vanishes as $n \to + \infty$.

  Concerning~\eqref{eq:7c}, $n$ does not appear in the first addend,
  while for the second one we get
  \begin{align*}
    &- \int_{\reali_+} \left(u_n (T,x) - k\right)^\pm \phi (T, x) \d{x}
    \\
    \leq \
    &
      - \int_{\reali_+} \left(u (T,x) - k\right)^\pm \phi (T, x) \d{x}
      +
      \int_{\reali_+} \left[\left(u (T,x) - k\right)^\pm
      - \left(u_n (T,x) - k\right)^\pm \right] \phi (T, x) \d{x}
    \\
    \leq \
    &
      - \int_{\reali_+} \left(u (T,x) - k\right)^\pm \phi (T, x) \d{x}
      +
      \int_{\reali_+} \modulo{u (T,x) - u_n (T,x)} \, \phi (T, x) \d{x},
  \end{align*}
  the second term vanishing as $n \to + \infty$. Finally, by the
  properties of the approximating sequence $v_n$ we know that
  $\norma{v_n}_{\L\infty ([0,t];\reali)}\leq \norma{v}_{\L\infty
    ([0,t];\reali)}$, and this concludes the proof.

  \paragraph{2. $\L1$--Lipschitz continuity in time.} For all
  $t_1, \, t_2 \in [0,T]$, by Point 2.~in Proposition~\ref{prop:2part},
  \begin{align*}
    \norma{u (t_1)- u (t_2)}_{\L1 (\reali_+;\reali)}
    = \
    & \lim_{n \to +\infty}
      \norma{u_n (t_1)- u_n (t_2)}_{\L1 (\reali_+;\reali)}
    \\
    \leq \
    & \tv (t_1 \vee t_2, u_o, u_b) \,
      \norma{v}_{\L\infty ([0,t_1 \vee t_2];\reali)}\,
      \norma{g'}_{\L\infty (\mathcal{U};\reali)}\,
      \modulo{t_1 - t_2},
  \end{align*}
  where $\tv (t,u_o,u_b)$ is as in~\eqref{eq:4} and
  $\mathcal{U} = \mathcal{U} (u_o, {u_b}_{[0,t_1\vee t_2]})$ is as
  in~\eqref{eq:19}.

  \paragraph{3. Total variation estimate. } For all $t \in [0,T]$, the
  lower semicontinuity of the total variation and Point 3.~in
  Proposition~\ref{prop:2part} yield
  \begin{displaymath}
    \tv \left(u (t)\right)
    \leq
    \lim_{n \to +\infty} \tv\left( u_n (t)\right)
    \leq
    \tv (t, u_o, u_b),
  \end{displaymath}
  proving Point 3.

  \paragraph{4. $\L1$--Lipschitz continuity on initial and boundary
    data. } Let $\tilde u_n$ be the solution to~\eqref{eq:un}
  corresponding to the initial datum $\tilde u_o$ and to the boundary
  datum $\tilde u_b$. Analogously as above, the sequence $\tilde u_n$
  converges in $\C0 ([0,T];\L1 (\reali_+;\reali))$ to a function
  $\tilde u$, which is a solution to~\eqref{eq:1part}.

  By Proposition~\ref{prop:nonso} and by the properties of $v_n$, for
  all $t \in [0,T]$,
  \begin{align*}
    \norma{u (t) - \tilde u (t)}_{\L1 (\reali_+;\reali)}
    = \
    & \lim_{n \to + \infty} \norma{u_n (t) - \tilde u_n (t)}_{\L1 (\reali_+;\reali)}
    \\
    \leq \
    & \norma{u_o - \tilde u_o}_{\L1 (\reali_+; \reali)}
      +
      \norma{v}_{\L\infty ([0,t];\reali)} \,
      \norma{g'}_{\L\infty (\mathcal{U};\reali)} \,
      \norma{u_b - \tilde u_b}_{\L1 ([0,t]; \reali)},
  \end{align*}
  where
  $\mathcal{U} = \mathcal{U} ({u_b}_{[0,t]}, {\tilde u}_{b[0,t]})$ is
  as in~\eqref{eq:19}, proving Point 4 and thus also the uniqueness
  of the solution to~\eqref{eq:1part}.

  \paragraph{5. $\L1$--stability with respect to $\boldsymbol{v}$ and $\boldsymbol{g}$. }
  Approximating $\tilde v$ as at the beginning of the proof yields the
  sequence $\tilde v_n$. Call $\tilde u_n$ the solution to the
  IBVP~\eqref{eq:un} corresponding to the flux
  $\tilde v_n \, \tilde g$. As above, the sequence $\tilde u_n$
  converges in $\C0 ([0,T];\L1 (\reali_+;\reali))$ to a function
  $\tilde u$, which is a solution to~\eqref{eq:1part}, corresponding
  to the flux $\tilde v \, \tilde g$.

  By Theorem~\ref{thm:3part}, for all $t \in [0,T]$
  \begin{align*}
    \norma{u (t) - \tilde u (t)}_{\L1 (\reali_+;\reali)}
    = \
    & \lim_{n \to + \infty} \norma{u_n (t) - \tilde u_n (t)}_{\L1 (\reali_+;\reali)}
    \\
    \leq \
    & \tv (t, u_o, u_b)
      \left(
      A \, t \, \norma{g' - \tilde g'}_{\L\infty (\mathcal{U};\reali)}
      +
      B  \, \norma{v-\tilde v}_{\L1 ([0,t];\reali)}
      \right),
  \end{align*}
  where $\tv (t,u_o,u_b)$ is defined in~\eqref{eq:4},
  $\mathcal{U} = \mathcal{U} (u_o,, {u_b}_{\vert [0,t]})$ is as
  in~\eqref{eq:19} and $A$, $B$ are as in~\eqref{eq:15}, completing
  the proof.
\end{proofof}

\noindent\textbf{Acknowledgement:} The present work was supported by the PRIN 2015 project \emph{Hyperbolic Systems of Conservation Laws
  and Fluid Dynamics: Analysis and Applications} and by the
INDAM--GNAMPA 2017 project \emph{Conservation Laws: from Theory to
  Technology} and by the MATHTECH project funded by CNR--INDAM. Part
of this work was accomplished while the authors were visiting the
Mittag--Leffler Institut.

{

  \small

  \bibliography{ColomboRossi}

  \bibliographystyle{abbrv}

}

\end{document}